\documentclass{article}

\usepackage[a4paper, total={6in, 8in}]{geometry} % Adjust margins
\usepackage[T1]{fontenc}
\usepackage[utf8]{inputenc}
\usepackage{comment}
\usepackage{amsfonts}
\usepackage{amsmath}
\usepackage{amsthm}
\usepackage{tikz}
\usepackage{mathabx}
\usepackage{comment}
\usepackage{graphicx,subcaption}
\usepackage{wrapfig}
\usepackage{biblatex}

\usepackage{tabstackengine}
\PassOptionsToPackage{linktocpage}{hyperref}

\usepackage{titletoc}
\usepackage{tocloft}

\usepackage{tikz}

\usetikzlibrary{positioning} 
\usetikzlibrary{calc}
\usetikzlibrary{arrows,shapes,backgrounds}

\newtheorem{theorem}{Theorem}[section]

\newtheorem{remark}[theorem]{Remark}
\newtheorem{lemma}[theorem]{Lemma}

\author{Gaspard Ohlmann \\ University of Basel, \emph{gaspard.ohlmann@unibas.ch}}
\title{A Study on the Well-Posedness of 1D Energy-Critical Half-Wave Maps Equations}
\date{}
\begin{document}
	
	\maketitle
	
%\begin{abstract}

%\end{abstract}

\begin{center} \textbf{Abstract} \end{center}

\emph{
    In this article, we study the well-posedness of the energy-critical half-wave maps equation (HWM) in dimension $1$. The half-wave maps equation emerges from the continuum limit of the Haldane Shastry spin chains and has been shown to arise as the continuum limit of Calogero-Moser classical spin systems. In higher dimension $d\geq 5$, it has been shown that (HWM) is well-posed in \cite{krieger2017small}. This result has been improved in \cite{kiesenhofer2021small} to $d = 4$ but the Strichartz estimate on which the argument is built no longer holds for smaller dimensions. A Lax-pair structure has been revealed for this equation in \cite{gerard2018lax}, indicating complete integrability and the fact that rational solutions stay rational for all time. The well-posedness of the (HWM) equation in lower dimensions remains an open problem. Here, we show the well-posedness of (HWM) in the rational case for finite times with separated poles, and for large and infinite times with distinct speeds of propagation.
}


\tableofcontents

\newpage

\section*{Introduction}

\subsection*{Context}

This article is devoted to the study of the well-posedness of the half-wave maps equation (HWM) in spatial dimension $1$.

\begin{equation}
    m_t = m \times |\nabla | m,
\end{equation}

where $m: \mathbb{R}\times \mathbb{R} \to \mathbb{S}^2$ maps the real line into the sphere $\mathbb{S}^2\subseteq \mathbb{R}^3$, and $|\nabla|$ is the pseudo-linear operator defined using the Fourier Transform as
\begin{equation}
    \widehat{|\nabla| u} (\xi)  = |\xi| \hat u(\xi),
\end{equation}

or equivalently, using the Hilbert transform, as 

\begin{equation}
    |\nabla| m = \mathcal{H} (\partial_x m),~ \mathcal{H}(m)(x) = \frac{1}{\pi} \text{p.v.} \int_{-\infty}^{\infty} \frac{m(\xi)}{x-\xi}d\xi.
\end{equation}

The half-wave maps equation arises as a continuum limit of a classical version of the Haldane Shastry spin chain, as studied in \cite{zhou2015solitons}. Their numerical experiment suggested the existence of true multi-soliton solutions. In \cite{lenzmann2020derivation}, E. Lenzmann proved that the half-wave maps equation arises as an effective equation in the continuum limit of completely integrable Calogero-Moser classical spin systems with inverse square $1/r^2$. The half-wave maps equation also arises as a limit case of a spin generalization of the Benjamin-Ono equation (sBO), as shown in \cite{berntson2022spin}. We study here the well-posedness of the equation. More precisely, we investigate whether the Cauchy problem

\begin{equation}
\left\{
    \begin{aligned}
    &m_t(t,x) = m(t,x) \times |\nabla| m(t,x), \\
    &m(0,x) = f(x),
    \end{aligned}
    \right.
\end{equation}

is well-posed for $f \in \dot H^{1/2}(\mathbb{R})$. In particular, we show under the assumptions that the spectrum is non-degenerate, meaning that the asymptotic propagation speeds are different, that the problem is well-posed for large times and does not blow up at $+\infty$. We also show, under the assumption that the poles are separated, that a blow-up can not occur at finite time. At the time of writing, the well-posedness of the half-wave maps equation is still a completely open problem, as stated in \cite{krieger2017small}, \cite{kiesenhofer2021small} or \cite{lenzmann2018short}. For higher dimensions, J. Krieger and Y. Sire show in \cite{krieger2017small} that the half-wave maps equation is well-posed for $d\geq 5$ provided that the initial condition is sufficiently small in a suitable Besov space. This result has been later improved by A. Kiesenhofer and J. Krieger in \cite{kiesenhofer2021small}, which shows that the equation is well-posed in dimension $4$ provided that the initial condition is smooth and sufficiently small in a Besov space. 

In our situation, however, the method is unfortunately not applicable, as the crucial $L_t^2 L_x^\infty$ Strichartz estimate is no longer available. For the half-wave maps, the formally conserved quantity reads

\begin{equation}
E[\mathbf{u}]=\frac{1}{2} \int_{\mathbb{R}^d} \mathbf{u} \cdot|\nabla| \mathbf{u} d x=c_d \iint_{\mathbb{R}^d \times \mathbb{R}^d} \frac{|\mathbf{u}(x)-\mathbf{u}(y)|^2}{|x-y|^2} d x d y,
\end{equation}

which makes this problem \textbf{energy critical} in dimension $n=1$, implying that the solutions could in principle exhibit strange behavior, such as global existence for some initial data and blow-up for others. 

Finally, we would like to mention that weak solutions exist for large initial data in $\dot H^1 \cap \dot H^{1/2} (\mathbb{R})$, under the assumption that the function is smooth and constant outside a compact domain, as proven by Y. Liu in \cite{liu2023global}. This is achieved by considering a regularized equation with a dispersive term. 

In this article, we are interested in particular in rational solutions of (HWM).

For a function $m:\mathbb{R}\times \mathbb{R} \to \mathbb{S}^2$, we can make the pole ansatz (as in \cite{matsuno2022integrability}, \cite{airault1977rational} or \cite{lenzmann2018short})

\begin{equation}
    m(t,x) = m_0 + \sum_{j=1}^N \frac{s_j(t)}{x-x_j(t)} + \sum_{j=1}^N \frac{\bar s_j(t)}{x-\bar x_j(t)}.
\end{equation}

In order to ensure that $m$ stays on the sphere on the whole line, we also need the constraints 
\begin{equation}
\mathbf{s}_{j, 0}^2=0, \quad \mathbf{s}_{j, 0} \cdot\left(\mathrm{im}_0-\sum_{k \neq j}^N \frac{\mathbf{s}_{k, 0}}{x_{j, 0}-x_{k, 0}}+\sum_{k=1}^N \frac{\mathbf{s}_{k, 0}^*}{x_{j, 0}-x_{k, 0}^*}\right)=0, \quad(j=1,2, \ldots, N).
\end{equation}

Those solutions are called rational solutions. 

Recently, a Lax pair structure has been found for the half-wave maps equation in \cite{gerard2018lax} by E. Lenzmann and P. G\'erard. Among other important properties, this implies in particular that the system is completely integrable and that an infinite number of quantities are preserved. Using Kronecker's theorem for finite rank Hankel operators, they also show that a rational solution of (HWM) at a given time stays rational for all time. In \cite{berntson2020multi}, B. K. Berntson, R. Klabbers and E. Langmann show that the (HWM) is satisfied if and only fine the poles and spins evolve according to the dynamics of an exactly solvable spin Calogero-Moser system given by

\begin{equation}
\dot{\mathbf{s}}_j=-2 \sum_{k \neq j}^N \frac{\mathbf{s}_j \wedge \mathbf{s}_k}{\left(a_j-a_k\right)^2},
\end{equation}

and 

\begin{equation}
\ddot{a}_j=4 \sum_{k \neq j}^N \frac{\mathbf{s}_j \cdot \mathbf{s}_k}{\left(a_j-a_k\right)^3}.
\end{equation}

Using these equations, Y. Matsuno has provided in \cite{matsuno2022integrability} explicit formulas for preserved quantities, as well as an expression for the asymptotic behavior in the two solitons case $(N=2)$. In \cite{berntson2020multi}, the authors present exact solutions, displaying solitary spin excitations moving with different velocities and interacting in a non-trivial way. Finally, in \cite{lenzmann2018energy}, E. Lenzmann and A. Schikorra give a complete classification of the traveling solitary waves of finite energy of the form $u(t,x)=Q_v(x-vt)$. These waves are rational and can be expressed using a \emph{finite Blaschke product} and a rotation $R \in SO(3)$. 

Beyond this, the study of the half-wave maps equation (HWM) remains a nascent field, and we are only at the initial stages of understanding. For instance, solutions for $N$ other than $2$ are unknown except for very specific forms, and it remains unclear whether blow-up can occur in finite or infinite time. In this article, we provide proof that the energy-critical half-wave maps equation is well-posed for large times in the non-degenerate case, and that (HWM) is well-posed under the assumption that the poles are separated. 

We start by dealing with the case $N=2$ with a new method through which we corroborate previously established results. In this scenario, the analytic process is simplified, given that the equations and constraints culminate in explicit formulas, enabling the straightforward derivation of results. 
We then consider the large-time and infinite-time case for $N$ solitons. The key argument is to take advantage of the constraints and preserved quantities, which we choose to write using a fractional calculus formula, as we did in \cite{ohlmann2021ill} to show the ill-posedness of the quasi-linear wave equation in dimension $2$ in $H^{7/4}$. We only consider the case where one pole is going toward the real axis. The results and inequalities established are later re-used for the general case.
Later, we show in section \ref{sec:N,sep} the well-posedness of (HWM) under the assumption that the poles are separated. To achieve this, we first perform an analysis similar to what we do in section \ref{N,diff,1pole} for a more specific case, but some of the estimates we had no longer hold as we reason for fixed times. In particular, the time evolution equations of the spins and poles do not lead to a vanishing factor anymore so the estimates are more intricate. We deal with the fully general case using an expression for the preserved quantity provided in \cite{matsuno2022integrability}.
Finally, we provide in section \ref{sec:turbu} explicit bounds for the spins and the imaginary part of the poles, quantifying the uniform non-turbulence in our situations, and complete the argument for the absence of infinite-time blow-up in the non-degenerate case.

\subsection*{Definitions and main result}

We recall the half-wave maps equation

\begin{equation}\tag{HWM}\label{HWM}
    m_t (t,x) = m(t,x) \times |\nabla| m(t,x).
\end{equation}
We consider functions in spatial dimension $1$ onto the sphere:
\begin{equation}
    m:\mathbb{R} \times \mathbb{R} \to \mathbb{S}^2,
\end{equation}
belonging in the $\dot H^{1/2}(\mathbb{R})$ Sobolev space at all time, i.e.
\begin{equation}
    ||m(t,\cdot)||_{\dot H^{\frac{1}{2}}(\mathbb{R})}^2 = \int_{\xi \in \mathbb{R}} |\xi| |\hat m(t,\xi)|^2 d\xi < \infty,
\end{equation}
where $\hat f$ denotes the standard Fourier Transform of $f$.

For such functions, we make the pole ansatz \cite{airault1977rational}

\begin{equation}\label{expressionSol}\tag{PA}
m(x,t)= m_0 +i \sum_{j=1}^N \frac{s_j(t)}{x-x_j(t)} - i \sum_{j=1}^N \frac{\bar s_j(t)}{x-\bar x_j(t)},
\end{equation}	

where the spins $s_j$ belong to $\mathbb{C}^3$ and the poles $x_j$ belong to the upper-half plane $\mathbb{C}^+$, and $\bar{z}$ stands for the complex conjugate of $z$. We denote by $\cdot$ the product $s \cdot v = s_1 v_1 + s_2 v_2 + s_3 v_3$.

$m$ solves \eqref{HWM} if the spins $s_j$ and poles $x_j$ satisfy the system

\begin{equation}\tag{CM}\label{CM}
\left\{
\begin{aligned}
& \dot{\mathbf{s}}_j(t)=-2 \sum_{k \neq j}^N \frac{\mathbf{s}_j(t) \times \mathbf{s}_k(t)}{\left(x_j(t)-x_k(t)\right)^2}, \quad(j=1,2, \ldots, N), \\
& \ddot{x}_j(t)=4 \sum_{k \neq j}^N \frac{\mathbf{s}_j(t) \cdot \mathbf{s}_k(t)}{\left(x_j(t)-x_k(t)\right)^3}, \quad(j=1,2, \ldots, N),
\end{aligned}
\right.
\end{equation}
with the initial conditions $\mathbf{s}_j(0)=\mathbf{s}_{j, 0}, x_j(0)=x_{j, 0}$ and
\begin{equation}\tag{IC}\label{conditions}
\dot{x}_j(0)=\frac{\mathbf{s}_{j, 0} \times \mathbf{\bar s}_{j, 0}}{\mathbf{s}_{j, 0} \cdot \mathbf{\bar s}_{j, 0}} \cdot\left(\mathrm{im}_0-\sum_{k \neq j}^N \frac{\mathbf{s}_{k, 0}}{x_{j, 0}-x_{k, 0}}+\sum_{k=1}^N \frac{\mathbf{\bar s}_{k, 0}}{x_{j, 0}-\bar x_{k, 0}}\right), \quad(j=1,2, \ldots, N) .
\end{equation}

This system of equation \eqref{CM} for $s_j$ and $x_j$ admits a Lax pair \cite{gerard2018lax} \cite{matsuno2022integrability} of the form
\begin{equation}
    \dot L = [B,L] = BL - LB,
\end{equation}
where $L$ is defined as
\begin{equation}
L=\left(l_{j k}\right)_{1 \leq j, k \leq N}, \quad l_{j k}=\delta_{j k} \dot{x}_j+\left(1-\delta_{j k}\right) \frac{\epsilon_{j k} \sqrt{2 s_j \cdot s_k}}{x_j-x_k},
\end{equation}

and $L$ is defined as 
\begin{equation}
B=\left(b_{j k}\right)_{1 \leq j, k \leq N}, \quad b_{j k}=\left(1-\delta_{j k}\right) \frac{\epsilon_{j k} \sqrt{2 s_j \cdot s_k}}{\left(x_j-x_k\right)^2},
\end{equation}

where $\delta_{j k}$ is the usual Kronechker's function and $\epsilon_{j k}$ is an anti-symmetric symbol defined by $\epsilon_{j k}= - \epsilon_{k j}$ and $\epsilon_{j k}^2 = 1 - \delta_{j k}$. In particular, this implies the very important other formulation for the time evolution of the poles
\begin{equation}
    \dot X = L + [B,X],
\end{equation}

where $X$ is a diagonal matrix containing the poles $x_j$ at coordinate $(j,j)$. By defining $\dot U = BU$ and considering $J = J(t) = U^{-1}(t) X(t) U(t)$, it follows 

\begin{equation}
    J(t) = L(0) t + X(0),
\end{equation}

which implies

\begin{equation}
    X(t) = U(t) \left[ L(0) t + X(0) \right] U^{-1} (t).
\end{equation}

This means that the spectrum of $X(t)$ is the same as the spectrum of $L(0) t + X(0)$. As a consequence, we have that for large times, the poles $x_i$ are equivalent to $v_i t$, where we denoted by $v_i$ the eigenvalues of $L$. It can be shown that the eigenvalues $v_i$ are real and that they satisfy $|v_i|\leq 1$. Throughout this article, we will say that the spectrum is \emph{non-degenerate} when $v_i \neq v_j$ for any $i \neq j$. 

We now state our main theorem.
\begin{theorem}
    Let $m$ be a solution of \eqref{HWM} of the form \eqref{expressionSol}. 
    \begin{itemize}
        \item[(i)] Assuming $v_i\neq v_j$ for $i \neq j$, then for $t>T$ large enough, including $t=+\infty$, the spins are globally bounded and the poles stay away from the real axis. 
        \item[(ii)] Assuming that the poles stay away from each other, the spins are globally bounded and the poles stay away from the real axis.
    \end{itemize}
\end{theorem}

Before this, we describe the dynamics in the two solitons case for large and infinite times.

\section{Two solitons case, infinite time}
	
In this section, we show the absence of blow-up at infinite time in the non-degenerate case ($v_1\neq v_2$) for the case $N=2$. 

When $N=2$, (\ref{expressionSol}) becomes 

\begin{equation}\label{expressionSol2}
m(x,t)= m_0 +i \frac{s_1(t)}{x-x_1(t)}- i \frac{\bar s_1(t)}{x-\bar x_1(t)}+i \frac{s_2(t)}{x-x_2(t)} - i \frac{\bar s_2(t)}{x-\bar x_2(t)},
\end{equation}

and (\ref{CM})

\begin{equation}\label{evolution2}
\begin{aligned}
&\dot s_1(t) = -2  \frac{s_1(t) \times s_2(t)}{\left(x_1(t) - x_2(t)\right)^2}, \\
&\dot s_2(t) = -2  \frac{s_2(t) \times s_1(t)}{\left(x_2(t) - x_1(t)\right)^2}, \\
&\ddot x_1(t) = 4 \frac{s_1(t) \cdot s_2(t)}{\left(x_1(t) - x_2(t)\right)^3}\\
&\ddot x_j(t) = 4 \frac{s_2(t) \cdot s_1(t)}{\left(x_2(t) - x_1(t)\right)^3}.
\end{aligned}
\end{equation}

The main result of this section is stated in the next theorem.

\begin{theorem}
	If we assume that $Im(x_1)(t) \rightarrow 0$ as $t\rightarrow \infty$, and that $Im(x_1)(t),~Im(x_2)(t) \neq 0$ for any $t$, we have a contradiction. 
\end{theorem}

\begin{proof}(a)
	We first note that (\ref{evolution2}) gives $\dot s_1(t) + \dot s_2(t) = 0$, and $\ddot x_1(t) + \ddot x_2(t) = 0.$ 
	
From \cite{matsuno2022integrability}, we know that there exist $\alpha_1, \alpha_2$ such that
\begin{equation}\label{Mats}
	\begin{aligned}
v_1 + v_2 &= \dot x_1 + \dot x_2, ~ (i)\\
\alpha_1+\alpha_2 &= -(x_1(0)+x_2(0)), ~ (ii) \\
x_1(t) &= v_1 t -\alpha_1 + o(1),~ (iii) \\
x_2(t) &= v_2 t -\alpha_2 + o(1).~ (iv)\\
	\end{aligned}
\end{equation}

This already gives that $Im(x_1)$ and $Im(x_2)$ have a limit. In particular, $Im(x_1)$ and $Im(x_2)$ are bounded.

We look again at (\ref{expressionSol}). We see that

\begin{equation}\label{1}
	i\left( \frac{s_1(t)}{x - x_1(t)} - \frac{\bar s_1(t)}{x - \bar x_1(t)} \right) = -2  \cdot Im\left( \frac{s_1(t)}{x - x_1(t)} \right),
\end{equation}

and 

\begin{equation}\label{2}
\frac{s_1(t)}{x-x_1(t)} = \frac{\left[ (x-Re(x_1)(t)) + i Im(x_1)(t) \right] s_1(t)}{(x-Re(x_1)(t))^2 + (Im(x_1)(t))^2}.
\end{equation}

Hence, plugging (\ref{1}) and (\ref{2}) into (\ref{expressionSol}) yields

\begin{multline}\label{expressionSolfactor}
m(x,t)= m_0 - 2 Im\left[ \frac{\left[ (x-Re(x_1)(t)) + i Im(x_1)(t) \right] s_1(t)}{(x-Re(x_1)(t))^2 + (Im(x_1)(t))^2} \right] \\- 2 Im\left[ \frac{\left[ (x-Re(x_2)(t)) + i Im(x_2)(t) \right] s_2(t)}{(x-Re(x_2)(t))^2 + (Im(x_2)(t))^2} \right].
\end{multline}

Now, if we consider the choice of $x=Re(x_1)$, we obtain

\begin{equation}\label{expressionSolfactor1}
m(x,t)= m_0 - 2 Im\left( \frac{i s_1(t)}{Im(x_1)} \right) - 2 Im\left( \frac{s_2(t)}{x - x_2(t)} \right).
\end{equation}

Now, because $\dot s_1(t)+\dot s_2(t)=0$, we have that $s_2(t)=a-s_1(t)$. We can then modify (\ref{expressionSolfactor1}) using the fact that

\begin{equation}
Im\left( \frac{s_2(t)}{x - x_2(t)} \right)= Im\left(\frac{a}{x-x_2(t)} \right) - Im\left(\frac{s_1(t)}{x-x_2(t)} \right).
\end{equation}

Finally, with $x=Re(x_1)$,

\begin{multline}\label{expressionSolfactor2}
m(x,t)= m_0 - 2 Im\left( \frac{i s_1(t)}{Im(x_1)} \right) - 2 Im\left(\frac{a}{x-x_2(t)} \right) - 2 Im\left(\frac{s_1(t)}{x-x_2(t)} \right)\\
=m_0 + 2 \frac{Re(s_1)(t)}{Im(x_1)} - 2 Im\left(\frac{a}{x-x_2(t)} \right) - 2 Im\left(\frac{s_1(t)}{x-x_2(t)} \right).
\end{multline}

Now, it is easy (really) to see that for $t$ big enough,

\begin{equation}
|\delta|=\left| \frac{a}{x-x_2} \right| \leq \frac{C}{t^2}.
\end{equation}

For the third term, we write

\begin{equation}
	Im\left(\frac{s_1(t)}{x-x_2(t)} \right)=\frac{Re(s_1)Im(x_2)+Im(s_1)Re(x_1-x_2)}{\left( Re(x_1)-Re(x_2) \right)^2 + Im(x_2)^2 }
\end{equation}

Now, (\ref{expressionSolfactor2}) becomes

\begin{multline}\label{expressionSolfactor3}
m(x,t)= m_0 -2 \delta + Re(s_1)\cdot \left( \frac{2}{Im(x_1)} -\frac{2 \cdot Im(x_2)}{\left( Re(x_1)-Re(x_2) \right)^2 + Im(x_2)^2 } \right) \\+ \frac{Im(s_1) Re(x_1-x_2)}{\left( Re(x_1)-Re(x_2) \right)^2 + Im(x_2)^2 }\\
= m_0 -2 \delta + Re(s_1) \cdot \beta + Im(s_1) \cdot \gamma,
\end{multline}

where it is easy to see that $\delta\rightarrow 0$, $\beta \rightarrow +\infty$, $\gamma\rightarrow 0$. Our ultimate goal for now is to show that $Re(s_1)$ and $Im(s_1)$ must go toward 0.

From (\ref{expressionSolfactor3}), we get that either $Re(s_1) \rightarrow 0$, or $Im(s_1) \rightarrow +\infty$ (quickly). We want to rule out the second case. To do so, we will chose another value for $x$, close to $Re(x_1)$.

We consider (\ref{expressionSolfactor}) and choose $x=Re(x_1)+Im(x_1) \in \mathbb{R}$. We obtain

\begin{multline}\label{otherx}
m(x,t)= m_0 - 2 Im\left[ \frac{\left[ Im(x_1)(t)) + i Im(x_1)(t) \right] s_1(t)}{2\cdot (Im(x_1)(t))^2} \right] \\- 2 Im\left(\frac{a}{x-x_2(t)} \right) - 2 Im\left(\frac{s_1(t)}{x-x_2(t)} \right).
\end{multline}

Now, we have

\begin{equation}
- 2 Im\left[ \frac{\left[ Im(x_1)(t)) + i Im(x_1)(t) \right] s_1(t)}{2\cdot (Im(x_1)(t))^2} \right] = - \frac{Im(s_1)+Re(s_1)}{Im(x_1)}.
\end{equation}

Again, we can see that (the $'x'$ is not the same, hence the change of letter)

\begin{equation}
	|\mu| =\left| \frac{a}{x-x_2} \right| \leq \frac{C}{t^2}.
\end{equation}

Now, for the third term, we obtain

\begin{equation}
	-2 Im\frac{s_1}{x-s_2(t)} = \frac{Re(s_1) Im(x_2) + Im(s_1) \left( Re(x_1)-Re(x_2) + Im(x_1) \right)}{\left( Re(x_1)-Re(x_2) + Im(x_1) \right)^2 + Im(x_2)^2 }.
\end{equation}

Hence, up to some simplifications, (\ref{otherx}) becomes

\begin{equation}\label{otherx2}
m(x,t) = m_0 - \mu - Re(s_1) \nu + Im(s_1) \xi,	
\end{equation}

where $\mu\rightarrow 0$, $\nu\rightarrow +\infty$, $\xi \rightarrow +\infty$. 
But now, if we are in the second case previously described, $Im(s_1)$ is of a bigger order than $Re(s_1)$, but $\nu$ and $\xi$ are of the same order, so it can not be of norm 1.

This means $Re(s_1) \rightarrow 0$, and hence $Im(s_1)\rightarrow 0$ from (\ref{otherx2}). Using (\ref{evolution2}), we obtain that $Re(s_2) \rightarrow l_1$ and $Im(s_2)\rightarrow l_2$. In particular, there exists $c>0$ such that

\begin{equation}
	|s_2| \leq C,~ \forall t
\end{equation}

Now, we will use the fact that

\begin{equation}
	\dot s_1(t) = - 2 \frac{s_1 (t) \times s_2(t)}{\left( x_1(t)-x_2(t) \right)^2},
\end{equation}

to obtain a contradiction. Since we have 

\begin{equation}
	(x_1-x_2)^2 \sim (v_1-v_2)^2 \cdot t^2,
\end{equation}

We obtain 

\begin{equation}
\left| \dot{|s_1|} \right| \leq \left| \dot s_1(t) \right| \leq \frac{C\cdot |s_1(t)|}{t^2}.
\end{equation}

Now, we remark that $\dot{|s_1|}<0$, so $\left| \dot{|s_1|} \right|=-\dot{|s_1|}$. We then have

\begin{equation}
\ln(|s_1|(t_0))-\ln(|s_1|(t_1))   =\int_{t=t_0}^{t_1} \frac{-\dot{|s_1|}}{|s_1|} \leq \int_{t=t_0}^{t_1} \frac{C}{t^2} = C \left( \frac{1}{t_0}-\frac{1}{t_1} \right).
\end{equation}

Now, taking $t_1\rightarrow +\infty$ yields a contradiction.

\end{proof}

\section{No infinite time blow-up with different speeds, specific case}\label{N,diff,1pole}
	
In this section, we will show the absence of an infinite time blow-up in the non-degenerate case, i.e. $v_i\neq v_j$ for $i\neq j$, when $m$ is the sum of $N$ solitons. To do so, we will show that the imaginary part of the role can not converge toward zero. In fact, we can easily adapt the proof to show that there exists $C>0$ such that the imaginary parts of the poles always stay bigger than $C$.

We state our result in the following theorem.

\begin{theorem}\label{majeur}
Let $m:\mathbb{R}\times \mathbb{R} \rightarrow S\subset \mathbb{R}^3$ be a solution of the Cauchy problem
 	
\begin{equation}
\left\{
\begin{aligned}
&m_t(x,t) = m(x,t) \times |\nabla| m(x,t),~ ~ \text{ for all } (x,t) \in \mathbb{R}^2, \\
&m(x,0)= m_0 +i \sum_{j=1}^N \frac{s_j^0}{x-x_j^0} - i \sum_{j=1}^N \frac{\bar s_j^0}{x-\bar x_j^0}, ~ ~ \text{ for all } x \in \mathbb{R}.
\end{aligned}
\right.
\end{equation}
	
We know that $m$ is of the form 
\begin{equation}\label{expressionSoln}
m(x,t)= m_0 +i \sum_{j=1}^N \frac{s_j(t)}{x-x_j(t)} - i \sum_{j=1}^N \frac{\bar s_j(t)}{x-\bar x_j(t)}.
\end{equation}	

We make the following assumption.

\begin{enumerate}
	\item Non-degenerate: for all $i,j$ such that $i\neq j$, $v_i\neq v_j$.
	\item No finite time blow-up: For all $t \in \mathbb{R}$ and $i\in \{1,...,N\}$, we have $Im(x_i(t)) \neq 0 $.
	\item Only one pole causing the blow-up: There exists $C>0$ such that for any $i\neq 1$ and $t\in \mathbb{R}$, we have $\left| Im(x_i(t)) \right| >C $.
\end{enumerate}

Then, assuming $Im(x_1) \rightarrow 0$ yields a contradiction.

\end{theorem}

We first study the long-term behavior of the spins. In the following lemma, we show that the spins are bounded.

\begin{lemma}\label{lemma1}
	There exists a constant $C$ such that for any $t\in \mathbb{R}$, we have
	\begin{equation}
		|s_j(t)| \leq C.
	\end{equation}
\end{lemma}

\begin{proof}

We now want to show that $s_j$ are bounded. To do so, we will study the expression of $m$ given by (\ref{expressionSoln}), and put it in perspective with condition (\ref{conditions}). We apply $x=Re(x_{i_0})$ to (\ref{expressionSoln}) and obtain

\begin{equation}
m(Re(x_{i_0}))=m_0 +i \sum_{j=1}^N \frac{s_j(t)}{Re(x_{i_0})-x_j(t)} - i \sum_{j=1}^N \frac{\bar s_j(t)}{Re(x_{i_0})-\bar x_j(t)}
\end{equation}

We study $s_{i_0} \cdot i m(Re(i_{x_0}))$ and obtain

\begin{equation}\label{thisisanew}
s_{i_0} \cdot i m(Re(i_{x_0})) = s_{i_0} \cdot \left(i m_0 - \sum_{j=1}^N \frac{s_j(t)}{Re(x_{i_0})-x_j(t)} + \sum_{j=1}^N \frac{\bar s_j(t)}{Re(x_{i_0})-\bar x_j(t)}\right).
\end{equation}

Now, we plug in condition (\ref{conditions}) for $i=i_0$ and combine it with (\ref{thisisanew}) to obtain

\begin{multline}\label{Dprime}
s_{i_0} \cdot i m(Re(i_{x_0})) \\= s_{i_0} \cdot \left( \sum_{j\neq i_0} \frac{s_j(t)}{x_{i_0}-x_j(t)} - \frac{s_j(t)}{Re(x_{i_0})-x_j(t)} +\sum_{j\neq i_0} \frac{\bar s_j(t)}{Re(x_{i_0})-\bar x_j(t)} - \frac{\bar s_j(t)}{x_{i_0}-\bar x_j(t)} \right) \\
+ \frac{|s_{i_0}|^2}{Re(x_{i_0})-\bar x_{i_0}(t)} - \frac{ |s_{i_0}|^2}{x_{i_0}-\bar x_{i_0}(t)}
\end{multline}

For the last difference term of (\ref{Dprime}), we quickly obtain that
\begin{equation}\label{Cprime}
\frac{|s_{i_0}|^2}{Re(x_{i_0})-\bar x_{i_0}(t)} - \frac{ |s_{i_0}|^2}{x_{i_0}-\bar x_{i_0}(t)} = \frac{1}{2} \cdot \frac{|s_{i_0}|^2}{Im(x_{i_0})}.
\end{equation}

For the first difference term of (\ref{Dprime}), we obtain

\begin{equation}\label{A}
\frac{s_j(t)}{x_{i_0}-x_j(t)} - \frac{s_j(t)}{Re(x_{i_0})-x_j(t)} = \frac{- Im(x_{i_0})\cdot s_j}{(Re(x_{i_0})-x_j(t))\cdot (x_{i_0}-x_j(t)) }
\end{equation}

The second term being almost the conjugate of the first one, we obtain

\begin{equation}\label{B}
-\frac{\bar s_j(t)}{x_{i_0}-\bar x_j(t)} + \frac{\bar s_j(t)}{Re(x_{i_0})-\bar x_j(t)} = \frac{ Im(x_{i_0})\cdot \bar s_j}{(Re(x_{i_0})-\bar x_j(t))\cdot (x_{i_0}-\bar x_j(t)) }
\end{equation}

We will now provide upper-bounds since we want these difference terms to be small. (\ref{A}) leads to

\begin{equation}\label{Aprime}
\left|\frac{s_j(t)}{x_{i_0}-x_j(t)} - \frac{s_j(t)}{Re(x_{i_0})-x_j(t)}\right| \leq \frac{|Im(x_{i_0})|\cdot |s_j|}{|v_{i_0}-v_j|^2 \cdot t^2},
\end{equation}

and (\ref{B}) leads to

\begin{equation}\label{Bprime}
\left|\frac{\bar s_j(t)}{x_{i_0}-\bar x_j(t)} + \frac{\bar s_j(t)}{Re(x_{i_0})-\bar x_j(t)}\right| \leq \frac{|Im(x_{i_0})|\cdot |s_j|}{|v_{i_0}-v_j|^2 \cdot t^2}.
\end{equation}

Overall, since $|s_{i_0} \cdot i m| \leq |s_{i_0}|$, we obtain using (\ref{Cprime}), (\ref{Aprime}) and (\ref{Bprime}) and the triangular inequality in (\ref{Dprime}) that

\begin{equation}
\frac{|s_{i_0}|^2}{|Im(x_{i_0})|} \leq 1 + 2 \sum_{j\neq i_{0}} \frac{|Im(x_{i_0})|\cdot |s_j| \cdot |s_{i_0}|}{|v_{i_0}-v_j|^2 \cdot t^2}.
\end{equation}

Hence, \textbf{if $Im(x_{i_0})$ is bounded}, we have by considering the supremum of $|s_{k}|$ over all $k$ that $|s_j|$ is bounded for every $j$.

\begin{comment}
By a similar argument, we obtain that

\begin{equation}
\frac{|s_{i_0}|^2}{|Im(x_{i_0})|} \geq 1 - 2 \sum_{j\neq i_{0}} \frac{|Im(x_{i_0})|\cdot |s_j| \cdot |s_{i_0}|}{|v_{i_0}-v_j|^2 \cdot t^2},
\end{equation}

which leads to $s_1 \rightarrow 0$. 

\end{comment}

\end{proof}
	
We first show that the spin $s_1$ has to converge to $0$. To do so, we will use the fact that the $H^{1/2}$ norm of the solution is constant.
	
\begin{lemma}
	Under the same assumptions, we have
	\begin{equation}
		|s_1|\rightarrow 0,\text{ as } t\rightarrow \infty.
	\end{equation}
\end{lemma}

\begin{proof}
	
	In this section, we will use the fact that the imaginary parts of the poles are bounded.
	Indeed, we have
	
	\begin{equation}\label{untrucconstant}
		\sum_{i} \dot x_i = \sum_i v_i \in \mathbb{R}.
	\end{equation}
	
	we obtain from (\ref{untrucconstant}) that 
	
	\begin{equation}\label{untrucconstant2}
		\sum_{i} Im(x_i) = cste.
	\end{equation}
	
	Now, (\ref{untrucconstant2}) together with the fact that $Im(x_i)\geq 0$ for all $i$ yields 
	
	\begin{equation}
		|Im(x_i)|\leq \sum_i Im(x_i(0)).
	\end{equation}

	Now, denoting $\mathcal{F}$ the Fourier transform, we provide the following expressions.

	\begin{multline}\label{four1}
		||m(\cdot,t)||_{H^{1/2}}^2 = || |\xi| \cdot \mathcal{F}(m(\cdot,t)) ||_{L^2(\mathbb{R})}^2 = < |\xi|^{1/2} \mathcal{F}(m), |\xi|^{1/2} \mathcal{F}(m)>_{L^2(\mathbb{R})} \\= < |\xi| \mathcal{F}(m), \mathcal{F}(m)>_{L^2(\mathbb{R})} = <|\nabla| m(\cdot,t),m(\cdot,t)>_{L^2(\mathbb{R})} = \int_{x=-\infty}^{+\infty} |\nabla| m(x,t) \cdot m(x,t) dx
	\end{multline}
	
	Another way to express this norm is

	\begin{multline}\label{H12}
		||m(\cdot,t)||_{H^{1/2}}^2 = \int_{x=-\infty}^{+\infty} \left| \int_{y=-\infty}^{+\infty} \frac{\partial m}{\partial x}(y,t) \cdot \frac{1}{\sqrt{|x-y|}}  dy \right|^2 dx\\
	=\int_{x=-\infty}^{+\infty} \left[  \int_{y=-\infty}^{+\infty} \frac{\partial m}{\partial x}(y,t)\cdot \frac{1}{\sqrt{|x-y|}}  dy \cdot \int_{z=-\infty}^{+\infty} \frac{\partial m}{\partial x}(z,t)\cdot \frac{1}{\sqrt{|x-z|}} dz \right] dx.
	\end{multline}	

	We will now assume that $Im(x_1) \rightarrow 0$ as $t\rightarrow \tilde{t}$, where $\tilde{t} \in (0,\infty]$.
	
	We recall that $m$ is a sum of solitons, hence 

	\begin{multline}\label{H122}
		||m(\cdot,t)||_{H^{1/2}}^2
		\\= \int_{x=-\infty}^{+\infty} \left| \int_{y=-\infty}^{+\infty} \frac{\partial }{\partial x} \left( m_0 + \sum_{j=1}^N \left( i 	\frac{s_j}{y-x_j} - i \frac{\bar s_j}{y-\bar x_j} \right) \right) \cdot \frac{1}{\sqrt{|x-y|}}  dy \right|^2 dx.\\
	\end{multline}	
	
	We define the $N+1$ following terms ($A_1$ is defined separately because it will not have a symmetric role)
	
	\begin{equation}
		\begin{aligned}
			A_0(x)&=\int_{y=-\infty}^{+\infty} \frac{\partial }{\partial x}m_0(y,t) \cdot \frac{1}{\sqrt{|y-x|}}dy = 0,\\
			A_1(x)&=\int_{y=-\infty}^{+\infty} \frac{\partial }{\partial x} \left( i \frac{s_1}{x-x_1} - i \frac{\bar s_1}{x-\bar x_1} \right)(y,t) \cdot \frac{1}{\sqrt{|y-x|}}dy,\\
			A_j(x)&=\int_{y=-\infty}^{+\infty} \frac{\partial }{\partial x} \left( i \frac{s_j}{x-x_j} - i \frac{\bar s_j}{x-\bar x_j} \right)(y,t) \cdot \frac{1}{\sqrt{|y-x|}}dy,~ j\neq 1.
		\end{aligned}
	\end{equation}
	
	Our goal will be to show that $A_1$ is much bigger than $A_j$, $j\neq 1$ when we are getting closer to $\tilde{t}$, except if $s_1\rightarrow 0$. We look at $A_1$ first. If $s_1\rightarrow l=(l_1,l_2) \in \mathbb{C} $ where $l\neq 0$, then there exists $t_0$ sufficiently close to $\tilde{t}$ such that $s_1\in [l_1/2,2\cdot l_1] \times [l_2/2,2\cdot l_2]$. If $|l| \rightarrow \infty$, we can chose this $t_0$ such that $|l|>M$. Also, if $|l| \rightarrow \infty$, since $s_1\cdot s_1=0$, we can chose $t_0$ such that we have $|l_1|>M$ and $|l_2|>M$. we will make a slight abuse of notation and always write $s_1\in [l_1/2,2\cdot l_1] \times [l_2/2,2\cdot l_2]$. 

	We compute and obtain

	\begin{multline}\label{derivlol}
	\frac{\partial }{\partial x} \left( i \frac{s_1}{x-x_1} - i \frac{\bar s_1}{x-\bar x_1} \right) 
	= i \frac{2 \bar s_1}{ (x-\bar x_1)^2 } - i \frac{s_1}{(x-x_1)^2} = 2 \cdot Re \left( i \frac{s_1}{\left( x-x_1\right)^2}  \right)\\
	= 2 \cdot Re\left( i \frac{Re(s_1) + i Im(s_1)}{[(x-Re(x_1))^2-Im(x_1)^2] - 2i(x-Re(x_1)) Im(s_1)} \right)\\
	= \frac{-2 Im(s_1) [(x-Re(x_1))^2 - Im(x_1)^2] -4 Re(s_1)(x-Re(x_1))Im(x_1)}{[(x-Re(x_1))^2 - Im(x_1)^2]^2 + 4 (x-Re(x_1))^2 Im(x_1)^2 }\\
		= \frac{-2 Im(s_1) [(x-Re(x_1))^2 - Im(x_1)^2] -4 Re(s_1)(x-Re(x_1))Im(x_1)}{[(x-Re(x_1))^2 + Im(x_1)^2]^2 }\\
	\end{multline}
	
	We recall that we have $s_1^2=0$, which means
	\begin{multline}
		\langle Re(s_1) + i Im(s_1) , Re(s_1)+ i Im(s_1) \rangle \\
		= |Re(s_1)|^2 - |Im(s_1)|^2 + 2 i \langle Re(s_1),Im(s_1) \rangle =0.
	\end{multline}
	
	In particular, we obtain
	
	\begin{equation}
		|Re(s_1)|^2 = |Im(s_1)|^2,~ \langle Re(s_1),Im(s_1) \rangle =0.
	\end{equation}
	
	Hence, for $a,b \in \mathbb{R}$, we have
	
	\begin{equation}
		\left| a Re(s_1) + b Im(s_1) \right| \geq |a| \cdot |Re(s_1)|.
	\end{equation}
	
	Using this property, we obtain,
	
	\begin{multline}
		|A_1(x)| = \left| \int_{y=-\infty}^{+\infty} \frac{\partial }{\partial x} \left( i \frac{s_1}{x-x_1} - i \frac{\bar s_1}{x-\bar x_1} \right)(y,t) \cdot \frac{1}{\sqrt{|y-x|}}dy \right| \\
		= \bigg| Im(s_1) \int_{y=-\infty}^{+\infty} \frac{-2[(y-Re(x_1))^2 - Im(x_1)^2]}{[(y-Re(x_1))^2 + Im(x_1)^2]^2 } \frac{1}{\sqrt{|y-x|}} dy \\
		+ Re(s_1) \int_{y=-\infty}^{+\infty} \frac{-4(y-Re(x_1))Im(x_1)}{[(y-Re(x_1))^2 + Im(x_1)^2]^2 } \frac{1}{\sqrt{|y-x|}} dy \bigg|\\
		\geq \left| \int_{y=-\infty}^{+\infty} \frac{-4(y-Re(x_1))Im(x_1)}{[(y-Re(x_1))^2 + Im(x_1)^2]^2 } \frac{1}{\sqrt{|y-x|}} dy \right| |Re(s_1)|
	\end{multline}
	
	Now, we assume that $x > Re(x_1)$. We will detail later why it is enough. We denote $X = x - Re(x_1)$, and $Y = y - Re(x_1)$. We have
	\begin{equation}
		|A_1(X)| \geq \left| \int_{Y=-\infty}^{+\infty} \frac{-4YIm(x_1)}{[Y^2 + Im(x_1)^2]^2 } \frac{1}{\sqrt{|Y-X|}} dy \right| |Re(s_1)|
	\end{equation}
	
	We study the inner integral. For $X>0$, we have
	
	\begin{multline}
		\int_{Y=-\infty}^{+\infty} \frac{4Y|Im(x_1)|}{[Y^2 + Im(x_1)^2]^2 } \frac{1}{\sqrt{|Y-X|}} dy \\
		=\int_{Y=0}^{+\infty} \frac{4Y|Im(x_1)|}{[Y^2 + Im(x_1)^2]^2 } \left( \frac{1}{\sqrt{|Y-X|}} -\frac{1}{\sqrt{Y+X}} \right) dy \\
		\geq \int_{Y=0}^{X} \frac{4Y|Im(x_1)|}{[Y^2 + Im(x_1)^2]^2 } \left( \frac{Y}{\sqrt{X^2-Y^2} \left( 2 \sqrt{X+Y} \right)} \right) dy \\
		+ \int_{Y=X}^{\infty} \frac{4Y|Im(x_1)|}{[Y^2 + Im(x_1)^2]^2 } \left( \frac{X}{\sqrt{Y^2-X^2} \left( 2 \sqrt{X+Y} \right)} \right) dy = C+D.
	\end{multline}
	
	Now, we provide a lower bound for $C$. We consider $X \in (0,|Im(x_1)|)$,
	
	\begin{multline}
		D \geq \int_{Y=X}^\infty \frac{4 Y |Im(x_1)|}{[Y^2 + Im(x_1)^2]^2 } \left( \frac{X}{\sqrt{Y^2} \left( 4 \sqrt{Y} \right)} \right) dy \\
		\geq \int_{Y=X}^\infty \frac{4 |Im(x_1)|}{[Y^2 + Im(x_1)^2]^2 } \frac{X}{ 4 \sqrt{Y} }  dy \geq \int_{Y=X}^\infty \frac{C X |Im(x_1)|}{[Y^2 + Im(x_1)^2]^{2+1/4} } dy \\
		\geq C X |Im(x_1)| \int_{Y=|Im(x_1)|}^\infty \frac{1}{[Y^2 + Im(x_1)^2]^{2} } dy \geq \frac{C X}{|Im(x_1)|^2}.
	\end{multline}
	
	In particular, we obtain that for $x \in [0,|Im(x_1)|]$, we have
	
	\begin{equation}
		|A_1(X)| \geq |Re(s_1)| \frac{CX}{|Im(x_1)|^2}.
	\end{equation}
	
	Now, we find an upper bound for the other terms, i.e. $A_k$, for $k\neq 1$.
	
	We have
	
	\begin{equation}
		\left| 2 Re\left( i \frac{s_k}{(y-x_k)^2} \right) \right| \leq 2 \frac{|s_k|}{(y-Re(x_2))^2+Im(x_2)^2},
	\end{equation}
	
	and consequently
	
	\begin{multline}
		|A_k(x)| \leq \int_{y=-\infty}^\infty \frac{C }{(y-Re(x_2))^2 + Im(x_2)^2} \frac{1}{\sqrt{|x-y|}} dy \\
		\leq \int_{|y-x|\leq 1} \frac{C}{(y-Re(x_2))^2 + Im(x_2)^2} \frac{dy}{\sqrt{|x-y|}} + \int_{|y-x|\geq 1} \frac{C dy}{(y-Re(x_2))^2 + Im(x_2)^2} \\
		\leq Cst.
	\end{multline}
	
	Now, we can finally estimate $||m||_{1/2}^2$. We have,
	
	\begin{multline}
		||m||_{1/2}^2 \geq \int_{X=0}^{|Im(x_1)|} \left( \sum_{k=1}^N A_k(X) \right)^2 dx \\
		\geq \int_{X=0}^{|Im(x_1)|} \left[ A_1(X)^2 - |A_1(X)| \sum_{k=2}^N |A_k(X)| - \sum_{j,k\neq 1} |A_j(X)A_k(X)| \right] dX \\
		\geq \int_{X=0}^{|Im(x_1)|} A_1(X)^2 dX - C \int_{X=0}^{|Im(x_1)|} |A_1(X)| dX - D |Im(x_1)|.
	\end{multline}
	
	Since we have,
	
	\begin{equation}
		\int_{X=0}^{|Im(x_1)|} |Re(s_1)|^2 \frac{X^2}{|Im(x_1)|^4} dX \geq \frac{C}{|Im(x_1)|} |Re(s_1)|^2 \rightarrow \infty,
	\end{equation}
	
	and 
	
	\begin{multline}
		\int_{X=0}^{|Im(x_1)|} |A_1(X)| dX \leq \left( \int_{X=0}^{|Im(x_1)|} 1 dX \right)^{1/2} \left( \int_{X=0}^{|Im(x_1)|} |A_1(X)|^2 dX \right)^{1/2} \\
		\leq |Im(x_1)|^{1/2} \left( \int_{X=0}^{|Im(x_1)|} |A_1(X)|^2 dX \right)^{1/2} << \int_{X=0}^{|Im(x_1)|} |A_1(X)|^2 dX.
	\end{multline}
	
	Finally, we obtain if we assume $|Re(s_1)|$ does not converge to $0$ (up to a subsequence)
	
	\begin{equation}
		||m||_{1/2}^2 \rightarrow \infty.
	\end{equation}
	
	This is of course in contradiction with the conservation of $||m||_{1/2}$. Hence, for any $\varepsilon$, there exists no sequence in time $t_k$ such that $|Re(x_1)|\geq \varepsilon$ and $t_k\rightarrow \infty$. This leads to the fact that $|Re(s_1)| \rightarrow 0$. Since $|Re(s_1)|=|Im(s_1)|$, we obtain $s_1 \rightarrow 0$.
	
	\end{proof}
	
We are now able to make the proof of our main theorem.

\begin{proof}(of theorem \ref{majeur})	
	
	Now, we consider the time evolution of $s_j$ provided by

	\begin{equation}
	\dot s_j = -2 \sum_{k\neq j} \frac{s_j(t) \times s_k(t)}{ (x_j(t)-x_k(t))^2 }.
	\end{equation}

	In particular, with $j=1$, we obtain using the previous estimates that (using again the fact that $Im(x_j)$ is bounded)

	\begin{equation}
	\left| \dot s_1 \right| \leq C \cdot \frac{|s_1|}{t^2} .
	\end{equation}

	Using the same argument as previously, this is in contradiction with the fact that $s_1 \rightarrow 0$ as $t\rightarrow \infty$, and $s_1 \neq 0$ for all $t< \infty$.

\end{proof}

 \begin{remark}
     The same property holds if we assume that several poles are causing the blow-up. This can be shown by using the techniques that we present in the next section, the main argument remains the same. We briefly show at the end how to use what we have done to deduce the well-posedness in the general setup. 
 \end{remark}

\section{Absence of finite time blow-up for N solitons with separated poles}\label{sec:N,sep}

In this section, we show the absence of a finite time blow-up phenomenon under the assumptions that there exists no couple of indices $i$ and $j$ with $i\neq j$ such that $x_i-x_j\rightarrow 0$ as $t\rightarrow 0$. Up to a subsequence, we can then assume that there exists $c>0$ such that for any $t$ close enough to $t_0$, $|x_i-x_j|\geq c$ for $i\neq j$. 

We now state the main result of this section. 

\begin{theorem}
	
	Let $m:\mathbb{R}\times \mathbb{R} \rightarrow S\subset \mathbb{R}^3$ be a solution of the Cauchy problem
 	
	\begin{equation}
	\left\{
	\begin{aligned}
	&m_t(x,t) = m(x,t) \times |\nabla| m(x,t),~ ~ \text{ for all } (x,t) \in \mathbb{R}^2, \\
	&m(x,0)= m_0 +i \sum_{j=1}^N \frac{s_j^0}{x-x_j^0} - i \sum_{j=1}^N \frac{\bar s_j^0}{x-\bar x_j^0}, ~ ~ \text{ for all } x \in \mathbb{R}.
	\end{aligned}
	\right.
	\end{equation}
	
	We know that $m$ is of the form 
	\begin{equation}
	m(x,t)= m_0 +i \sum_{j=1}^N \frac{s_j(t)}{x-x_j(t)} - i \sum_{j=1}^N \frac{\bar s_j(t)}{x-\bar x_j(t)}.
	\end{equation}
Under the previously mentioned assumptions, there exists no $i$ and $t_0$ such that $Im(x_{i})\rightarrow 0$ as $t\rightarrow t_0$, even up to a subsequence.
\end{theorem}

\subsection{Case where only one pole is going toward the real axis, proof with an analytic expression for the energy norm}

We first look at the case where only one pole is going toward the real axis, the general case will be done in the next subsection.

We will now define two subsets of $\{1,\cdots,N \}$. We denote by $I$ the set of indices such that $j\in I$ implies that $s_j$ is not bounded near $t_0$. We also denote by $J$ the set of indices such that $j \in J$ implies that there exists a subsequence in time such that $Im(x_j)\rightarrow 0.$

Our proof will be by contradiction and divided into the following three lemmas. We assume that $J$ is not empty.

\begin{comment}

\begin{lemma}\label{lezerolemme2}
	For any time $t_0>0$, there exists a constant $C(t_0)<\infty$ such that for any $t<0$ and any $k$, we have
	\begin{equation}
		\left| Re(x_k(t)) \right|, \left| Im(x_k)\right| ,|x_k|\leq C(t_0). 
	\end{equation}
\end{lemma}

\end{comment}

\begin{lemma}\label{encoreencoreun}
	For any $1 \leq i \leq N$, we have that $|s_i|\leq \frac{C}{|Im(x_{i_1})|}$.
\end{lemma}

\begin{lemma}\label{encoreencoreun2}
For any $j\in J$, we have that $|s_j|\rightarrow 0$ up to a subsequence.
\end{lemma}

\begin{lemma}\label{encoreencoreun3}
	For any $1 \leq i \leq N$, we in fact have $|s_i|\leq C$.
\end{lemma}

From these three lemmas, we will deduce a contradiction. 

\begin{proof}(of lemma \ref{encoreencoreun})
		
The key fact for this proof will be that $m(x,t)$ has to be bounded for any $x$ and $t$. We will consider the boundedness of $m$ for $2\cdot |N|$ points: $\max(Re(x_{i_j})+j)$ for $1 \leq j \leq 2\cdot |N|$.

We have the following equality 

\begin{equation}\label{icilol2}
		|m(x,t)|= \left| m_0 + \sum_{k=1}^N \left( \frac{s_k(t)}{(x-x_k(t))} - \frac{\bar s_k}{(x-\bar x_k(t))} \right) \right| \leq C,
\end{equation}

From (\ref{icilol2}), we obtain the following system of linear equations 

\begin{equation}\label{unsystemelol}
	\left\{
	\begin{aligned}
	& p_1 = - \sum_{1\leq k\leq N} \frac{s_{k}}{(\max\left( Re(x_{j}) \right)+1-x_{k})} - \sum_{1\leq k\leq N} \frac{\bar s_{k}}{(\max\left( Re(x_{j}) \right)+1-\bar x_{k})} \\
	&[\cdots]\\
	& p_N = - \sum_{1\leq k\leq N} \frac{s_{k}}{(\max\left( Re(x_{j}) \right)+N-x_{k})} - \sum_{1\leq k\leq  N} \frac{\bar s_{k}}{(\max\left( Re(x_{j}) \right)+N-\bar x_{k})}\\
	\end{aligned}
	\right.
\end{equation}

where each $p_i$ is bounded. This means that we have

\begin{equation}\label{ecritmatrice}
	M S=P,~S=(s_i)_{1\leq i \leq N},~ |P|\leq C,
\end{equation}

where 

\begin{multline}
	M=\\
	\setlength\arraycolsep{0.3pt}
	%\scriptsize{
 \begin{pmatrix}
	\frac{1}{(\max\left( Re(x_{j}) \right)+1-x_{1})} & \cdots & \frac{1}{\max\left( Re(x_{j}) \right)+1-x_{N}} &\frac{1}{\max\left( Re(x_{i_j}) \right)+1-\bar x_{1}} & \cdots & \frac{1}{\max\left( Re(x_{j}) \right)+1-\bar x_{N}}\\
	\vdots & \ddots & \vdots & \vdots & \ddots & \vdots\\
	\frac{1}{\max\left( Re(x_{j}) \right)+2\cdot N-x_{1}} & \cdots & \frac{1}{\max\left( Re(x_{j}) \right)+2\cdot N-x_{N}} &\frac{1}{\max\left( Re(x_{j}) \right)+2\cdot N-\bar x_{1}} & \cdots & \frac{1}{\max\left( Re(x_{j}) \right)+2\cdot N-\bar x_{N}} \\ \end{pmatrix}
	%}
\end{multline}

Now, the determinant of the corresponding matrix can be computed as a Cauchy-determinant and is of form

\vspace{0.2cm}
\begin{center}
\resizebox{\linewidth}{!}{%
$
	\begin{aligned}
	\renewcommand{\arraystretch}{1.6}
	&\left| \begin{array}{cccccc}
	\frac{1}{(\max\left( Re(x_{j}) \right)+1-x_{1})} & \cdots & \frac{1}{\max\left( Re(x_{j}) \right)+1-x_{N}} &\frac{1}{\max\left( Re(x_{i_j}) \right)+1-\bar x_{1}} & \cdots & \frac{1}{\max\left( Re(x_{j}) \right)+1-\bar x_{N}}\\
	\vdots & \ddots & \vdots & \vdots & \ddots & \vdots\\
	\frac{1}{\max\left( Re(x_{j}) \right)+2\cdot N-x_{1}} & \cdots & \frac{1}{\max\left( Re(x_{i_j}) \right)+2\cdot N-x_{N}} &\frac{1}{\max\left( Re(x_{j}) \right)+2\cdot N-\bar x_{1}} & \cdots & \frac{1}{\max\left( Re(x_{j}) \right)+2\cdot N-\bar x_{N}} \\
	 \end{array} \right| \\
	&= \frac{\prod\limits_{i<j} (\max\left( Re(x_{j}) \right)+i-\max\left( Re(x_{j}) \right)-j)\prod\limits_{i<j} (\bar x_{i}-\bar x_{j})\prod\limits_{i<j} (x_{i}-x_{j})}{\prod\limits_{i,j} (\max\left( Re(x_{j}) \right)+i- \bar x_{j})\prod\limits_{i,j} (Re(x_{i}) - x_{j})} \\
	& =\frac{\prod\limits_{i<j} (i-j)\prod\limits_{i<j} (\bar x_{i}-\bar x_{j})\prod\limits_{i<j} (x_{i}-x_{j})}{\prod\limits_{i,j} (\max\left( Re(x_{j}) \right)+i- \bar x_{j})\prod\limits_{i,j} (Re(x_{i}) - x_{j})} \neq 0.
	\end{aligned}
$
}
\end{center}

Now, this determinant is not $0$ but it can be very small. We will look at the operator norm of the inverse of this matrix.

%We will now clarify why we obtain a contradiction.

We can compute the inverse $B$ of this matrix, and $B=(b_{i,j})_{i,j}$ is given by

\begin{equation}\label{desbavec1}
b_{ij} = \frac{ \prod_{k=1}^{2N} (z_j + y_k)(z_k + y_i) } {(z_j + y_i) \left( \prod_{k\neq j}  (z_j - z_k) \right) \left( \prod_{k\neq i} (y_i - y_k) \right) },
\end{equation}

where $z_j= \max(Re(x_{i_j}))+j$ for $1 \leq j \leq 2 N$, $y_j=z_{i_j}$ for $1 \leq j \leq N$ and $y_j=\bar z_{i_j}$ for $N + 1 \leq j \leq 2 N$. 

\begin{remark}
	Note that the $z_i$ coefficients are chosen, they are the points we apply our equality to, but the $y_j$ coefficients must be the poles and their conjugates.
\end{remark}

Now, we look at the different parts of the coefficients $b_{i,j}$. We have the following

\begin{equation}
	\begin{aligned}
		&|z_j + y_k |= \left| \max(Re(x_i))+j + Re(x_k) \pm i\cdot Im(x_k)\right| \leq C,\\
		&|z_k + y_i |= \left| \max(Re(x_t))+k + Re(x_i) \pm i\cdot Im(x_i)\right| \leq C,\\
		&|z_i + y_j |= \left| \max(Re(x_t))+i + Re(x_j) \pm i\cdot Im(x_j)\right| \geq C,\\
		&|z_j-z_k|=|j-k|\geq 1,\\
		&|y_i-y_k| \geq \nu, ~ \text{ for } |i-k| \neq N, \\
		&|y_i-y_{i+|I|}| \geq |Im(x_i)| \geq C, ~i\neq i_1 \\
		&|y_{i_1}-y_{i_1+|I|}| \geq C \cdot |Im(x_{i_1})|, ~ \text{ for } |i-k| = N,~i\neq i_1 \\
	\end{aligned}
\end{equation}

We assume that the $x_j$ are bounded and deal with the general case at the end of the proof.

Now, all the coefficients of the inverse matrix satisfy

\begin{equation}
	|b_{i,j}| \leq \frac{C}{|Im(x_{i_1})|}.
\end{equation}

This means that the operator norm of the inverse matrix, denoted by $M^{-1}$, satisfies

\begin{equation}\label{triplenorm}
	|||M^{-1}||| \leq \frac{C}{|Im(x_{i_1})|}.
\end{equation}

We now plug (\ref{triplenorm}) in (\ref{ecritmatrice}) and obtain

\begin{equation}
	|S|\leq |||M^{-1}||| \cdot P \leq \frac{C}{|Im(x_i)|},
\end{equation}

which means that for any $j$, $s_j$ satisfies

\begin{equation}
	|s_j| \leq \frac{C}{|Im(x_{i_1})|}.
\end{equation}

This concludes the proof of lemma \ref{encoreencoreun}.

\begin{remark}
Note that the estimation of the operator norm of the inverse matrix is arbitrarily big. It makes sense because the column corresponding to $x_i$ and the column corresponding to $\bar x_i$ are very close, so the matrix is almost not invertible. 

\end{remark}

We now deal with the general case, where $x_j$ are not necessarily bounded. The additional step is that we have to prove that $b_{i,j}$ satisfies

\begin{equation}
	b_{ij} = \frac{ \prod_{k=1}^{2N} (z_j + y_k)(z_k + y_i) } {(z_j + y_i) \left( \prod_{k\neq j}  (z_j - z_k) \right) \left( \prod_{k\neq i} (y_i - y_k) \right) } \leq \frac{C}{Im(x_{i_1})}. 
\end{equation}

In this case, taking $z_k=\max(Re(x_j))+k$ for $1\leq k \leq 2N$ is not going to work. To obtain a better lower bound for the coefficients of the matrix we take coefficients that are closer to the poles. Concretely, we take

\begin{equation}
	\left\{
	\begin{aligned}
		z_k &= -x_k - \frac{\eta}{3}, ~ 1\leq k \leq N,\\
		z_k &= -\bar x_{N-k} - \frac{2\cdot \eta}{3}, ~ N+1 \leq k \leq 2N.
	\end{aligned}
	\right.
\end{equation}

We now look at the different coefficients of $b_{i,j}$. We have (for $1\leq j \leq N$ and $1\leq i \leq N$, the other cases are similar)

\begin{equation} \label{74lol}
	\begin{aligned}
	&\prod_{k\neq j} |z_k-z_j| = \prod_{k=1, k\neq j}^{N} |x_k-x_j| \prod_{k=N+1}^{N} |\bar x_{k-N}- \frac{\eta}{3} - x_j|,\\
	&\prod_{k\neq i} |y_i-y_k| = \prod_{k=1,k\neq i}^N |x_i-x_k| \prod_{k=N+1}^{2N} |x_i-\bar x_{k-N}|,\\
	&|z_j+y_i| = |y_i-y_j-\frac{\eta}{3}|, \\
	&\prod_{k} |z_j+y_k| = \prod_{k=1}^N |x_k-x_j-\frac{\eta}{3}| \prod_{k=N+1}^{2N} |\bar x_{k-N} - x_j-\frac{\eta}{3}|\\
	&\prod_{k} |z_k+y_i| = \prod_{k=1}^N |x_i-x_k-\frac{\eta}{3}| \prod_{k=N+1}^{2N} |x_i - \bar x_{k-N} - \frac{\eta}{3}|.
	\end{aligned}
\end{equation}

This means that we have for $i,j \neq i_1$

\begin{multline}\label{75lol}
	|b_{i,j}| = \frac{1}{|y_i-y_j-\frac{\eta}{3}|} \cdot \frac{\prod_{k=1}^N |x_k-x_j-\frac{\eta}{3}| \prod_{k=N+1}^{2N} |\bar x_{k-N} - x_j-\frac{\eta}{3}|}{\prod_{k=1, k\neq j}^{N} |x_k-x_j| \prod_{k=N+1}^{N} |\bar x_{k-N}- \frac{\eta}{3} - x_j|} \\ 
	\cdot \frac{\prod_{k=1}^N |x_i-x_k-\frac{\eta}{3}| \prod_{k=N+1}^{2N} |x_i - \bar x_{k-N} - \frac{\eta}{3}|}{\prod_{k=1,k\neq i}^N |x_i-x_k| \prod_{k=N+1}^{2N} |x_i-\bar x_{k-N}|}.
\end{multline}

Now,

\begin{multline}\label{76lol}
\frac{\prod_{k=1}^N |x_k-x_j-\frac{\eta}{3}| \prod_{k=N+1}^{2N} |\bar x_{k-N} - x_j-\frac{\eta}{3}|}{\prod_{k=1, k\neq j}^{N} |x_k-x_j| \prod_{k=N+1}^{N} |\bar x_{k-N}- \frac{\eta}{3} - x_j|} \\
= |\frac{\eta}{3}| \cdot |Im(x_j)-\frac{\eta}{3}| \cdot \frac{\prod_{k\neq j}^N |x_k-x_j-\frac{\eta}{3}| \prod_{k=N+1,k\neq j}^{2N} |\bar x_{k-N} - x_j-\frac{\eta}{3}|}{\prod_{k=1, k\neq j}^{N} |x_k-x_j| \prod_{k=N+1}^{N} |\bar x_{k-N}- \frac{\eta}{3} - x_j|} \\
\leq C \frac{\prod_{k\neq j}^N |x_k-x_j-\frac{\eta}{3}|}{\prod_{k=1, k\neq j}^{N} |x_k-x_j|} \leq C.
\end{multline}

Also,

\begin{multline}\label{77lol}
\frac{\prod_{k=1}^N |x_i-x_k-\frac{\eta}{3}| \prod_{k=N+1}^{2N} |x_i - \bar x_{k-N} - \frac{\eta}{3}|}{\prod_{k=1,k\neq i}^N |x_i-x_k| \prod_{k=N+1}^{2N} |x_i-\bar x_{k-N}|}\\ 
= \frac{\eta}{3} \cdot |Im(x_i)-\frac{\eta}{3}| \frac{\prod_{k=1,k\neq i}^N |x_i-x_k-\frac{\eta}{3}|}{\prod_{k=1,k\neq i}^N |x_i-x_k|}  \cdot \frac{\prod_{k=N+1,k\neq }^{2N} |x_i - \bar x_{k-N} - \frac{\eta}{3}|}{\prod_{k=N+1}^{2N} |x_i-\bar x_{k-N}|} \\
\leq C 
\end{multline}

This means that we have

\begin{equation}
	|b_{i,j}| \leq C.
\end{equation}

Now, if $j=i_1$ and $i\neq i_1$, the factor $y_{i_1} - y_{N+i_1} = 2 Im(x_{i_1})$ appears at the bottom, which leads to the corresponding coefficient in the upper-bound of $b_{i,j}$. If $i=i_1$, we obtain the factor $Im(x_{i_1}) + \eta/3$ which does not change the upper-bound. Finally, when both $i$ and $j$ are equal to $i_1$, the factor is the same again.
	
Finally, we have

\begin{equation}
	|b_{i,j}| \leq \frac{C}{|Im(x_{i_1})|}.
\end{equation}

\end{proof}

We now go on with the proof of lemma \ref{encoreencoreun2}. 

\begin{proof}(of lemma \ref{encoreencoreun2})

The key part here will be the fact that the quantity $||m||_{1/2}= ||m(\cdot,t)||_{H^{1/2}(\mathbb{R})}$ is conserved, as well as the fact that the real parts of the poles are well separated.

Again, we will use the following formula for the energy norm of $m$ and define $A$ and $B$ accordingly

	\begin{multline}\label{H1222}
		||m(\cdot,t)||_{H^{1/2}} ^2
		\\= \int_{y=-\infty}^{+\infty} \left( \int_{x=-\infty}^{+\infty} \frac{\partial }{\partial x} \left( m_0 + \sum_{j=1}^N \left( i 	\frac{s_j}{x-x_j} - i \frac{\bar s_j}{x-\bar x_j} \right) \right) \cdot \frac{1}{\sqrt{|x-y|}}  dy \right)^2 dx.\\
	\end{multline}	

\begin{comment}
	We define the $N+1$ following terms
	
	\begin{equation}
		\begin{aligned}
			A_0(x)&=\int_{y=-\infty}^{+\infty} \frac{\partial }{\partial x}m_0(y,t) \cdot \frac{1}{\sqrt{|y-x|}}dy = 0,\\
			A_1(x)&=\int_{y=-\infty}^{+\infty} \frac{\partial }{\partial x} \left( i \frac{s_1}{x-x_1} - i \frac{\bar s_1}{x-\bar x_1} \right)(y,t) \cdot \frac{1}{\sqrt{|y-x|}}dy,\\
			A_j(x)&=\int_{y=-\infty}^{+\infty} \frac{\partial }{\partial x} \left( i \frac{s_j}{x-x_j} - i \frac{\bar s_j}{x-\bar x_j} \right)(y,t) \cdot \frac{1}{\sqrt{|y-x|}}dy,~ j\neq 1.
		\end{aligned}
	\end{equation}
\end{comment}

We define 

\begin{equation}\label{definitiondeA}
	A(x,y,t) = \frac{1}{\sqrt{|x-y|}} \left[ 2 Re \left( \frac{s_1}{(x-x_1)^2}\right) + \sum_{j\neq 1} 2 Re \left( \frac{s_j}{(x-x_j)^2} \right) \right],
\end{equation}

and 

\begin{equation}\label{definitiondeB}
	B(x,t)=2 Re \left( \frac{s_1}{(x-x_1)^2}\right) + \sum_{j\neq 1} 2 Re \left( \frac{s_j}{(x-x_j)^2} \right).
\end{equation}

Using (\ref{definitiondeA}) and (\ref{definitiondeB}), we can rewrite $||m||_{1/2}$ as 

\begin{equation}\label{reecriture}
	||m||_{1/2}^2 = \int_{y=-\infty}^{+\infty} \left(\int_{x=-\infty}^{+\infty} A(x,y,t) \right)^2 = \int_{y=-\infty}^{+\infty} \left(\int_{x=-\infty}^{+\infty} \frac{1}{\sqrt{|x-y|}} B(x,t) \right)^2.
\end{equation}

Now, there exists $\delta>0$ such that $|t-t_0|< \delta$, $|x-Re(x_1)| < \delta$ and $|Re(x_1)-y|\leq \delta$ imply the three next equalities. From now on, we will assume that $|s_1|$ does not converge toward $0$, and hence we will restrict ourselves to times $t$ close to $t_0$ such that $|s_1(t)|\geq \varepsilon>0$.

\begin{itemize}
\item $|s_i| \leq \frac{C}{|Im(x_{i_1})|}$, 
\item $|x_i|, |Re(x_i)|, |Im(x_i)| \leq C$.
\end{itemize}

Our plan is as follows. We will first explain what we can consider that $|x-x_1|\leq \eta/2$. Then, we find a lower bound for $A(Re(x_1),y,t)$ that converges toward $+\infty$. After this, we will find a $\delta_x$ such that $B(x,t)$ is close enough to $B(Re(x_1),t)$ when $|x-Re(x_1)|<\delta_x$. Then, we proceed with the integration and find that $||m||_{1/2}\rightarrow +\infty$ as $t\rightarrow 0$. 

We restrict ourselves to $y\in [Re(x_1)-Im(x_1),Re(x_1)+Im(x_1)]$. We will denote $x=z+Re(x_1)$, and $y=\tilde{y}+Re(x_1)$. 

We write 

\begin{equation}
B(x,t) = 2 Re \left( \frac{s_1}{(x-x_1)^2}\right) + \sum_{j\neq 1} 2 Re \left( \frac{s_j}{(x-x_j)^2} \right) = B_1(x,t) + B_2(x,t),
\end{equation}

and define

\begin{equation}\label{lesilol}
	\int_{x=-\infty}^{\infty} \left( \frac{1}{\sqrt{|x-y|}} B_1(x,t) +  \frac{1}{\sqrt{|x-y|}} B_2(x,t) \right)  =: I_1(y,t) + I_2(y,t).
\end{equation}

Now, our goal will be to show that $\left| \int \frac{1}{|x-y|} B_1 \right| \rightarrow \infty$ and that $\left| \int \frac{1}{|x-y|} B_2 \right| \leq C$. We will first find a lower bound for $I_1$. 

We give an expression for $B_1$ first.

\begin{multline}
	2 Re\left( \frac{s_1}{(x-x_1)^2} \right) = \frac{z^2 Re(s_1) - Im^2(x_1) Re(s_1) + 2\cdot z Im(x_1)Im(s_1)}{\left( z^2+Im(x_1)^2 \right)^2}\\
	=\left( \frac{z^2}{(z^2+Im(x_1)^2)^2} - \frac{Im(x_1)^2}{(z^2+Im(x_1)^2)^2} \right) Re(s_1) + \frac{2\cdot z Im(x_1)}{\left( z^2+Im(x_1)^2 \right)^2} Im(s_1) \\
	= (C-D) Re(s_1) + E Im(s_1).
\end{multline}

Now, we are in the following situation. $C$ and $D$ lead to unbounded integrals, but they could cancel out. In fact, we have that

\begin{equation}
\int_{z\in \mathbb{R}} \left(C - D\right) = \frac{\pi}{2\cdot Im(x_1)} - \frac{\pi}{2 \cdot Im(x_1)}=0,
\end{equation}

so they actually cancel out. However, here the weight function $\frac{1}{\sqrt{|x-y|}}$ is added, so we expect it to break the balance between the two terms. We have chosen $y$ so that the weight is big when the functions are the most different, which is when $z$ is close to $0$. 
We define $C_y = \frac{1}{\sqrt{|x-y|}} C$ and $D_y = \frac{1}{\sqrt{|x-y|}} D$. We will now try to find an upper bound for $C_y$ and a lower bound (in norm) for $D_y$. 

We have by the change of variable $z=Im(x_1)\cdot \omega$,

\begin{equation}
\int_{z\in \mathbb{R}} \frac{1}{\sqrt{|z-\tilde y|}} \frac{z^2}{(z^2+Im(x_1)^2)^2} = \frac{1}{|Im(x_1)|^{3/2}} \int_{\omega \in \mathbb{R}} \frac{1}{\sqrt{\left|\omega-\frac{\tilde y}{|Im(x_1)|}\right|}} \frac{\omega^2}{(1+\omega^2)^2}
\end{equation}

With $\tilde y = |Im(x_1)| \cdot v$, we have

\begin{equation}
	\int_{z\in \mathbb{R}} C(x,y,t) = \frac{1}{|Im(x_1)|^{3/2}} \cdot f(v), 
\end{equation}

where 

\begin{equation}
	f(v) = \int_{\omega \in \mathbb{R}} \frac{1}{\sqrt{\left| \omega-v \right|}} \frac{\omega^2}{(1+\omega^2)^2}.
\end{equation}

Now, we have that $f(0) = \frac{3\pi}{2^{3/2}}$, this comes from the fact that when $\tilde y=0$,

\begin{equation}
\int_{z\in \mathbb{R}} C_y(\tilde y=0) =  \int_{z\in \mathbb{R}} \frac{1}{|z|} \cdot \frac{z^2}{(z^2+Im(x_1)^2)^2} = \frac{\pi}{2^{3/2} |Im(x_1)|^{3/2}}.
\end{equation}

We now deal with $D$.

We have by the change of variable $z=Im(x_1)\cdot \omega$,

\begin{equation}
\int_{z\in \mathbb{R}} \frac{1}{\sqrt{|z-\tilde y|}} \frac{|Im(x_1)|^2}{(z^2+Im(x_1)^2)^2} = \frac{1}{|Im(x_1)|^{3/2}} \int_{\omega \in \mathbb{R}} \frac{1}{\sqrt{\left|\omega-\frac{\tilde y}{|Im(x_1)|}\right|}} \frac{1}{(1+\omega^2)^2}
\end{equation}

With $\tilde y = |Im(x_1)| \cdot v$, we have

\begin{equation}
	\int_{z\in \mathbb{R}} D(x,y,t) = \frac{1}{|Im(x_1)|^{3/2}} \cdot g(v), 
\end{equation}

where 

\begin{equation}
	f(v) = \int_{\omega \in \mathbb{R}} \frac{1}{\sqrt{\left| \omega-v \right|}} \frac{1}{(1+\omega^2)^2}.
\end{equation}

Now, we have that $f(0) = \frac{\pi}{2^{3/2}}$, this comes from the fact that when $\tilde y=0$,

\begin{equation}
\int_{z\in \mathbb{R}} C_y(\tilde y=0) =  \int_{z\in \mathbb{R}} \frac{1}{|z|} \cdot \frac{|Im(x_1)|^2}{(z^2+Im(x_1)^2)^2} = \frac{\pi}{2^{3/2} |Im(x_1)|^{3/2}}.
\end{equation}

Finally, we obtain

\begin{equation}
	2Re\left( \frac{s_1}{(x-x_1)^2} \right) = \left( \frac{f(v)-g(v)}{|Im(x_1)|^{3/2}} \right) Re(s_1) + E \cdot Im(s_1).
\end{equation}

By continuity of $f-g$, there exists $\delta$ such that $|v|\leq \delta$ implies $|f-g|\geq \frac{1}{10}$. Now, the previous equality is an equality between two vectors. The coefficient in front of $Re(s_1)$ is unbounded. We will use the fact that $Re(s_1)$ and $Im(s_1)$ are orthogonal to deduce that the whole vector is unbounded. We have

\begin{equation}
	Re(s_1) \cdot 2Re\left( \frac{s_1}{(x-x_1)^2} \right) = \left( \frac{f(v)-g(v)}{|Im(x_1)|^{3/2}} \right) \cdot |Re(s_1)|^2.
\end{equation}

This means that 

\begin{equation}
	\left| 2Re \left( \frac{s_1}{(x-x_1)^2} \right) \right| \geq \frac{1}{10\cdot |Im(x_1)|^{3/2}} \cdot |Re(s_1)|. 
\end{equation}

Now, since $|Re(s_1)| = |Im(s_1)|$, we obtain $|Re(s_1)| = \frac{1}{\sqrt{2}} |s_1|$.

This is valid for $\tilde y \in [-\delta\cdot |Im(x_1)|, \delta \cdot |Im(x_1)|]$ where $\delta$ is an absolute constant. We will now go on and compute an upper bound for $I_2$ (defined in (\ref{lesilol})). We have for $j\neq i_1$ that

\begin{equation}
	\left| 2 Re\left( \frac{s_j}{(x-x_j)^2} \right) \right| \leq |s_j| \cdot \frac{1}{|x-x_j|^2} \leq \frac{C}{|Im(x_1)|} \cdot \frac{1}{|(x-Re(x_j))^2 + Im(x_j)^2|}.
\end{equation}

Since $|Im(x_j)|$ is bounded from below, we obtain 

\begin{multline}
\left| \int_{x\in \mathbb{R}} \frac{1}{\sqrt{|x-y|}} 2 Re\left( \frac{s_j}{(x-x_j)^2} \right) \right| \\ \leq_{Holder} \frac{C}{|Im(x_1)|} \cdot \left( \int_{x\in\mathbb{R}} \frac{1}{|x-y|^{9/10}} \right)^{5/9} \\ \cdot \left( \int_{x\in \mathbb{R} }\frac{1}{|(x-Re(x_j))^2 + Im(x_j)^2|^{9/4}} \right)^{4/9} \leq \frac{C}{|Im(x_1)|}.
\end{multline}

Now, we have

\begin{equation}
\int_{y=Re(x_{i_1}) - \delta \cdot |Im(x_{i_1})|}^{Re(x_{i_1}) + \delta \cdot |Im(x_{i_1})|} \left( A(x,y,t) \right)^2 \geq 2 \delta \cdot |Im(x_1)| \cdot \frac{C^2}{|Im(x_1)|^{3}} \geq \frac{C}{|Im(x_1)|},
\end{equation}

which is what we wanted. Indeed, since we assumed that $|s_1|\geq \varepsilon$, we deduce that the integral diverges and hence $||m||_{1/2}$ diverges as well. This concludes the proof of lemma \ref{encoreencoreun2}.

\end{proof}

We finally go on with the proof of lemma \ref{encoreencoreun3}. The arguments are similar to the ones used in the proof of lemma \ref{encoreencoreun}, so we will focus on the elements of the proof that are not identical. It is based on the fact that $m(x,t)$ has to be bounded for any $x$ and $t$ and uses an invertible Cauchy matrix.

\begin{proof}

Since $s_1\rightarrow 0$, we obtain that $s_1$ is bounded when $t$ is close to $t_0$. Now, we write as in (\ref{icilol2}) 
	
	\begin{equation}\label{icilol22}
			|m(x,t)|= \left| m_0 + \sum_{k=1}^N \left( \frac{s_k(t)}{(x-x_k(t))} - \frac{\bar s_k}{(x-\bar x_k(t))} \right) \right| \leq C,
	\end{equation}
	
	and as in (\ref{unsystemelol}), we obtain the system of linear equations
	
	\begin{equation}\label{unsystemelol2}
		\left\{
		\begin{aligned}
		& p_1 = - \sum_{1\leq k\leq N} \frac{s_{k}}{(z_1-x_{k})} - \sum_{1\leq k\leq N} \frac{\bar s_{k}}{(z_1-\bar x_{k})} \\
		&[\cdots]\\
		& p_{2N} = - \sum_{1\leq k\leq N} \frac{s_{k}}{(z_{2N}-x_{k})} - \sum_{1\leq k\leq  N} \frac{\bar s_{k}}{(z_{2N}-\bar x_{k})}\\
		\end{aligned}
		\right.
	\end{equation}

	where each $p_i$ is bounded, and $z_N$ can be chosen in $\mathbb{R}$. If we want the matrix to be invertible, we want that $z_i \neq z_j$, and $z_i \neq x_j,\bar x_j$. Now, since $x_{i_1}$ is bounded, we can put the corresponding terms on the left-hand side and obtain (we remove the $i_1$ and the $N+i_1$ lines)

	\begin{equation}\label{unsystemelol3}
		\left\{
		\begin{aligned}
		& \tilde p_1 = - \sum_{k \neq i_1} \frac{s_{k}}{(z_1-x_{k})} - \sum_{k\neq i_1} \frac{\bar s_{k}}{(z_1-\bar x_{k})} \\
		&[\cdots]\\
		& \tilde p_{2N} = - \sum_{k \neq i_1} \frac{s_{k}}{(z_{2N}-x_{k})} - \sum_{k \neq i_1} \frac{\bar s_{k}}{(z_{2N}-\bar x_{k})}\\
		\end{aligned}
		\right.
	\end{equation}
	
	This is a system of the form $AS=\tilde{P}$, where $\tilde P$ is bounded. The coefficients of the inverse of $A$ are similar to the ones described in (\ref{desbavec1}), but without the term corresponding to $x_{i_1}$, i.e. (expression that makes sense only when $i,j\neq i_1,N+i_1$)

	\begin{equation}\label{desbavec12}
	b_{ij} = \frac{ \prod_{k=1,k\neq i_1, k\neq N+ i_1}^{2N} (z_j + y_k)(z_k + y_i) } {(z_j + y_i) \left(  \prod_{k=1,k\neq i_1, k\neq N+ i_1,k\neq j}^{2N}   (z_j - z_k) \right) \left( \prod_{k=1,k\neq i_1, k\neq N+ i_1,k\neq i}^{2N}  (y_i - y_k) \right) }.
	\end{equation}
	
	Now, we can for instance choose 
	
	\begin{equation}
	\begin{aligned}
	z_j = -x_j - \frac{\eta}{3}, ~ 1 \leq j \leq N,~ j\neq i_1\\
	z_j = -\bar x_j - \frac{2\cdot \eta}{3}, ~ N+1 \leq j \leq 2N,~ j\neq N+i_1,\\
	\end{aligned}
	\end{equation}
	
	and obtain by the same argument as in (\ref{74lol}), (\ref{75lol}), (\ref{76lol}) and  (\ref{77lol}) that
	
	\begin{equation}
	|b_{i,j}| \leq C.
	\end{equation}
	
	Hence, since $S = A^{-1} \cdot \tilde P$, we have that $s_k$ is bounded for $k\neq i_1$. This concludes the proof of lemma \ref{encoreencoreun3}.
	
\end{proof}

We now conclude. We have

\begin{equation}
	|\dot s_1(t)| = 4 \left| \sum_{k=1}^N \frac{s_1 \times s_k}{ (x_1-x_k)^2} \right| \leq |s_1| \cdot \frac{4\cdot N \cdot C}{\eta^2} \leq C'|s_1|.
\end{equation}

Now, this leads to 

\begin{equation}
	-C' \leq \frac{\dot{|s_1|}}{|s_1|} \leq C',
\end{equation}

so

\begin{equation}
	-C'(t-t')\leq \ln(|s_1(t)|)-\ln(|s_1(t')|) \leq C'(t-t').
\end{equation}

The left hand side is bounded as $t\rightarrow t_0$ while the middle term satisfies $\ln(|s_1(t)|)-\ln(|s_1(t')|) \rightarrow -\infty$ as $t\rightarrow t_0$, which is a contradiction.

\subsection{Case where several poles are going toward the real axis: proof with an algebraic expression of the energy norm}

In this subsection, we deal with the general case (under the assumption that the real parts of the poles are well separated). We will mainly focus on the elements of the proof that are different, and skip some details for the parts of the proof that are identical to what can be found in the previous section.

First, we state a slightly modified version of our three lemmas. We define $J$ as the set of indices such that $i\in J$ implies $Im(x_i)\rightarrow 0$ as $t \rightarrow t_0$.

\begin{lemma}\label{encoreencoreun23}
	For any $1 \leq i \leq N$, we have that $|s_i|\leq \frac{C}{\min |Im(x_{i_1})|}$.
\end{lemma}

\begin{lemma}\label{encoreencoreun234}
For any $j\in J$, we have that $|s_j|\rightarrow 0$ up to a subsequence.
\end{lemma}

\begin{lemma}\label{encoreencoreun334}
	For any $1 \leq i \leq N$, we in fact have $|s_i|\leq C$.
\end{lemma}

From these three lemmas, we will deduce a contradiction. 

We first start with the proof of lemma (\ref{encoreencoreun23}). The proof shares similarities with the proof of lemma \ref{encoreencoreun}, but is not identical.

\begin{proof}

Using the boundedness of $m(x,t)$ for any $x$ and $t$, and applying this fact for $2N$ different points $x=z_i$, $1\leq i \leq 2N$, we obtain a system of the form

	\begin{equation}\label{unsystemelol22}
		\left\{
		\begin{aligned}
		& p_1 = - \sum_{1\leq k\leq N} \frac{s_{k}}{(z_1-x_{k})} - \sum_{1\leq k\leq N} \frac{\bar s_{k}}{(z_1-\bar x_{k})} \\
		&[\cdots]\\
		& p_{2N} = - \sum_{1\leq k\leq N} \frac{s_{k}}{(z_{2N}-x_{k})} - \sum_{1\leq k\leq  N} \frac{\bar s_{k}}{(z_{2N}-\bar x_{k})}\\
		\end{aligned}
		\right.
	\end{equation}

Now, using the same arguments, we find that coefficients of the inverse of this Cauchy matrix are of the form

	\begin{equation}\label{desbavec123}
	b_{ij} = \frac{ \prod_{k=1}^{2N} (z_j + y_k)(z_k + y_i) } {(z_j + y_i) \left( \prod_{k\neq j}  (z_j - z_k) \right) \left( \prod_{k\neq i} (y_i - y_k) \right) },
	\end{equation}

where we have and chose for $z_j$

\begin{equation}
	\left\{
	\begin{aligned}
	y_j &= x_j,~ 1 \leq j \leq N, \\
	y_j &= \bar x_{j-N},~ N+1 \leq j \leq 2N, \\
	z_j &= -x_j - \frac{\eta}{3}, ~ 1 \leq j \leq N,\\
	z_j &= -\bar x_j - \frac{2\cdot \eta}{3}, ~ N+1 \leq j \leq 2N.\\
	\end{aligned}
	\right.
\end{equation}

Now, a similar argument as in (\ref{74lol}), (\ref{75lol}), (\ref{76lol}) and  (\ref{77lol}) gives that for $i\notin J$ when $i \leq N$ and $i-N \notin J$ when $i\geq N+1$, 

\begin{equation}
|b_{i,j}| \leq C.
\end{equation}

When $i$ is in $j$ or $i-N$ is in $J$, we have (assuming $i\leq N$ and $j\in J$),

\begin{equation} \label{74lol2}
	\begin{aligned}
	&\prod_{k\neq j} |z_k-z_j| = \prod_{k=1, k\neq j}^{N} |x_k-x_j| \prod_{k=N+1}^{N} |\bar x_{k-N}- \frac{\eta}{3} - x_j|,\\
	&\prod_{k\neq i} |y_i-y_k| = \prod_{k=1,k\neq i}^N |x_i-x_k| \prod_{k=N+1, k\neq N+i}^{2N} |x_i-\bar x_{k-N}| \cdot (2 Im(x_i)),\\
	&|z_j+y_i| = |y_i-y_j-\frac{\eta}{3}|, \\
	&\prod_{k} |z_j+y_k| = \prod_{k=1}^N |x_k-x_j-\frac{\eta}{3}| \prod_{k=N+1}^{2N} |\bar x_{k-N} - x_j-\frac{\eta}{3}|\\
	&\prod_{k} |z_k+y_i| = \prod_{k=1}^N |x_i-x_k-\frac{\eta}{3}| \prod_{k=N+1}^{2N} |x_i - \bar x_{k-N} - \frac{\eta}{3}|.
	\end{aligned}
\end{equation}

\begin{multline}\label{75lol2}
	|b_{i,j}| = \frac{1}{|y_i-y_j-\frac{\eta}{3}|} \cdot \frac{\prod_{k=1}^N |x_k-x_j-\frac{\eta}{3}| \prod_{k=N+1}^{2N} |\bar x_{k-N} - x_j-\frac{\eta}{3}|}{\prod_{k=1, k\neq j}^{N} |x_k-x_j| \prod_{k=N+1}^{N} |\bar x_{k-N}- \frac{\eta}{3} - x_j|} \\ 
\cdot \frac{1}{2Im(x_i)}	\cdot \frac{\prod_{k=1}^N |x_i-x_k-\frac{\eta}{3}| \prod_{k=N+1}^{2N} |x_i - \bar x_{k-N} - \frac{\eta}{3}|}{\prod_{k=1,k\neq i}^N |x_i-x_k| \prod_{k=N+1,k\neq N+i}^{2N} |x_i-\bar x_{k-N}|}.
\end{multline}

Now,

\begin{multline}\label{76lol2}
\frac{\prod_{k=1}^N |x_k-x_j-\frac{\eta}{3}| \prod_{k=N+1}^{2N} |\bar x_{k-N} - x_j-\frac{\eta}{3}|}{\prod_{k=1, k\neq j}^{N} |x_k-x_j| \prod_{k=N+1}^{N} |\bar x_{k-N}- \frac{\eta}{3} - x_j|} \\
= |\frac{\eta}{3}| \cdot |Im(x_j)-\frac{\eta}{3}| \cdot \frac{\prod_{k\neq j}^N |x_k-x_j-\frac{\eta}{3}| \prod_{k=N+1,k\neq j}^{2N} |\bar x_{k-N} - x_j-\frac{\eta}{3}|}{\prod_{k=1, k\neq j}^{N} |x_k-x_j| \prod_{k=N+1}^{N} |\bar x_{k-N}- \frac{\eta}{3} - x_j|} \\
\leq C \frac{\prod_{k\neq j}^N |x_k-x_j-\frac{\eta}{3}|}{\prod_{k=1, k\neq j}^{N} |x_k-x_j|} \leq C.
\end{multline}

Also,

\begin{multline}\label{77lol2}
\frac{\prod_{k=1}^N |x_i-x_k-\frac{\eta}{3}| \prod_{k=N+1}^{2N} |x_i - \bar x_{k-N} - \frac{\eta}{3}|}{\prod_{k=1,k\neq i}^N |x_i-x_k| \prod_{k=N+1,k\neq N+i}^{2N} |x_i-\bar x_{k-N}|}\\ 
= \frac{\eta}{3} \cdot |Im(x_i)-\frac{\eta}{3}| \frac{\prod_{k=1,k\neq i}^N |x_i-x_k-\frac{\eta}{3}|}{\prod_{k=1,k\neq i}^N |x_i-x_k|}  \cdot \frac{\prod_{k=N+1}^{2N} |x_i - \bar x_{k-N} - \frac{\eta}{3}|}{\prod_{k=N+1,k\neq N+i}^{2N} |x_i-\bar x_{k-N}|} \\
\leq C.
\end{multline}

We hence have

\begin{equation}
	|b_{i,j}| \leq \frac{C}{|Im(x_i)|}.
\end{equation}

Now, this gives that $|||A^{-1}||| \leq C \sum_{i,j} |b_{i,j}| \leq \frac{C}{\min |Im(x_i)|}$. Hence, we obtain that for every $i$, $|s_i| \leq \frac{C}{\min |Im(x_i)|}$.

\end{proof}

We now go on with the proof of lemma \ref{encoreencoreun234}. This time, we will use an algebraic expression for the energy norm, to show another way of dealing with this preserved quantity. 

\begin{proof}

We have from \cite{matsuno2022integrability} that the quantity

\begin{equation}
||m||_{1/2} = \mathcal{J} = 2 \pi \sum_{j,k=1}^N \frac{s_j \cdot \bar s_k}{(x_j-\bar x_k)^2},  
\end{equation}

is preserved. Now, in this expression, the terms 

\begin{equation}
J_k = \frac{s_k\cdot \bar s_k}{(2 Im (x_k))^2} = \frac{|s_k|^2}{4(Im(x_k))^2}
\end{equation}

appear. We will show that they are the leading terms. We have when $j,k\in J$

\begin{equation}
	\left| \frac{s_j \cdot s_k}{(x_j-\bar x_k)^2 } \right| \leq \frac{|s_j| \cdot |s_k|}{\eta^2} \leq \left( \frac{ |s_j|^2}{\eta^4} + \frac{ |s_k|^2}{\eta^4} \right).
\end{equation}

When $j\in J$ and $k \notin J$, we have by lemma \ref{encoreencoreun23}

\begin{equation}
	\left| \frac{s_j \cdot s_k}{(x_j-\bar x_k)^2 } \right| \leq \frac{|s_j| \cdot |s_k|}{\eta^2} \leq |s_j| \cdot \frac{C}{\min(Im(x_i))}.
\end{equation}

When $j\notin J$ and $k\in J$, we have

\begin{equation}
	\left| \frac{s_j \cdot s_k}{(x_j-\bar x_k)^2 } \right| \leq \frac{|s_j| \cdot |s_k|}{\eta^2} \leq |s_k| \cdot \frac{C}{\min(Im(x_i))},
\end{equation}

and when $j \notin J$ and $k \notin J$, we have

\begin{equation}
	\left| \frac{s_j \cdot s_k}{(x_j-\bar x_k)^2 } \right| \leq \frac{|s_j| \cdot |s_k|}{\eta^2} \leq \frac{C^2}{\min(Im(x_i))^2}.
\end{equation}

We now look at the sum. We have

\begin{multline}
	\sum_{j,k} \frac{s_j \cdot \bar s_k}{(x_j-\bar x_k)^2 } = 
	\sum_{j\in J} \sum_{k\in J} \frac{s_j \cdot \bar s_k}{(x_j-\bar x_k)^2 } + \sum_{j\notin J} \sum_{k\in J} \frac{s_j \cdot \bar  s_k}{(x_j-\bar x_k)^2 } \\ + \sum_{j\in J} \sum_{k\notin J} \frac{s_j \cdot \bar s_k}{(x_j-\bar x_k)^2 } + \sum_{j\notin J} \sum_{k\notin J} \frac{s_j \cdot \bar s_k}{(x_j-\bar x_k)^2 }  = A + B + C + D.
\end{multline}

Now, 

\begin{equation}
A+B+C \geq \sum_{j \in J} |s_j| \cdot \left(\frac{C_1\cdot |s_j|}{Im(x_j)^2} - \frac{C_2}{ \min(Im(x_j))} \right).
\end{equation}

We can easily obtain the better upper bound since the coefficient of the matrix $b_{i,j}$ only involves the imaginary part of the poles corresponding to the row $i$, we have

\begin{equation}
|s_i| \leq \frac{C}{|Im(x_i)|}.
\end{equation}

From this, we obtain that for $t$ close enough to $t_0$, we have

\begin{equation}
|A+B+C| \geq C \sum_{j\in J} \frac{|s_j|^2}{Im(x_j)^2}.
\end{equation}

If we have the boundedness of $A+B+C$, we hence have the result, because each $s_j$ for $j\in J$ will have to converge toward $0$ as $t\rightarrow t_0$. We could show the boundedness of $D$ using a better estimate on the $s_j$ for $j\notin J$, but instead, we will show that $D\geq 0$, which also yields the result since $A+B+C+D=const$.

We have that

\begin{equation}
	D = \sum_{j,k\notin J} \frac{s_j \cdot \bar s_k}{(x_j-\bar x_k)^2}.
\end{equation}

If we define $\tilde m$, the function corresponding to the poles that are not in $J$, namely

\begin{equation}
	\tilde m(x,t) = m_0 + \sum_{j\notin J} \frac{s_j}{x-x_j} + \frac{\bar s_j}{x-\bar x_j},
\end{equation}

we do not have the boundedness of $\tilde m$, nor that $m$ is a solution of the PDE, but we still have

\begin{equation}
	D=||\tilde m||_{1/2}^2 \geq 0.
\end{equation}

Now, we can conclude that

\begin{equation}
	s_j \rightarrow 0,~ j\in J.
\end{equation}

This concludes the proof of lemma \ref{encoreencoreun234}.

\end{proof}

We now go on with the proof of lemma \ref{encoreencoreun334}. Again, the idea of the proof is that the spins $s_j$ are bounded when $j\in J$, so we can move them to the bounded left-hand side and remove them from the system of linear equations.

\begin{proof}

	Since $s_j\rightarrow 0$ when $j\in J$, we obtain that $s_j$ is bounded when $t$ is close to $t_0$. Now, we write as in (\ref{icilol2}) 
	
		\begin{equation}\label{icilol223}
				|m(x,t)|= \left| m_0 + \sum_{k=1}^N \left( \frac{s_k(t)}{(x-x_k(t))} - \frac{\bar s_k}{(x-\bar x_k(t))} \right) \right| \leq C,
		\end{equation}
	
		and as in (\ref{unsystemelol}), we obtain the system of linear equations
	
		\begin{equation}\label{unsystemelol23}
			\left\{
			\begin{aligned}
			& p_1 = - \sum_{1\leq k\leq N} \frac{s_{k}}{(z_1-x_{k})} - \sum_{1\leq k\leq N} \frac{\bar s_{k}}{(z_1-\bar x_{k})} \\
			&[\cdots]\\
			& p_{2N} = - \sum_{1\leq k\leq N} \frac{s_{k}}{(z_{2N}-x_{k})} - \sum_{1\leq k\leq  N} \frac{\bar s_{k}}{(z_{2N}-\bar x_{k})}\\
			\end{aligned}
			\right.
		\end{equation}

		where each $p_i$ is bounded, and $z_N$ can be chosen in $\mathbb{R}$. If we want the matrix to be invertible, we want that $z_i \neq z_j$, and $z_i \neq x_j,\bar x_j$. Now, since $x_{j}$ is bounded for $j\in J$, we can put the corresponding terms on the left-hand side and obtain (we remove the $j$ and the $N+j$ lines for $j\in J$)

		\begin{equation}\label{unsystemelol4}
			\left\{
			\begin{aligned}
			& \tilde p_1 = - \sum_{k \neq i_1} \frac{s_{k}}{(z_1-x_{k})} - \sum_{k\neq i_1} \frac{\bar s_{k}}{(z_1-\bar x_{k})} \\
			&[\cdots]\\
			& \tilde p_{2N} = - \sum_{k \neq i_1} \frac{s_{k}}{(z_{2N}-x_{k})} - \sum_{k \neq i_1} \frac{\bar s_{k}}{(z_{2N}-\bar x_{k})}\\
			\end{aligned}
			\right.
		\end{equation}
	
		This is a system of the form $AS=\tilde{P}$, where $\tilde P$ is bounded. The coefficients of the inverse of $A$ are similar to the ones described in (\ref{desbavec1}), but without the term corresponding to $x_{i_1}$, i.e. (expression that makes sense only when $i,j\neq i_1,N+i_1$)

		\begin{equation}
		b_{ij} = \frac{ \prod_{k=1,k\notin J, k-N \notin J}^{2N} (z_j + y_k)(z_k + y_i) } {(z_j + y_i) \left(  \prod_{k=1,k\notin J, k-N \notin J,k\neq j}^{2N}   (z_j - z_k) \right) \left( \prod_{k=1,k\notin J, k-N \notin J,k\neq i}^{2N}  (y_i - y_k) \right) }.
		\end{equation}
	
		Now, we can for instance choose 
	
		\begin{equation}
		\begin{aligned}
		z_j = -x_j - \frac{\eta}{3}, ~ 1 \leq j \leq N,~ j\notin J\\
		z_j = -\bar x_j - \frac{2\cdot \eta}{3}, ~ N+1 \leq j \leq 2N,~ j-N \notin J,\\
		\end{aligned}
		\end{equation}
	
		and obtain by the same argument as in (\ref{74lol}), (\ref{75lol}), (\ref{76lol}) and  (\ref{77lol}) that
	
		\begin{equation}
		|b_{i,j}| \leq C.
		\end{equation}
	
		Hence, since $S = A^{-1} \cdot \tilde P$, we have that $s_k$ is bounded for $k\neq i_1$. This concludes the proof of lemma \ref{encoreencoreun3}.
	
\end{proof}

Now, we easily conclude by obtaining a contradiction with the time-evolution equation of the spin. We have for $j\in J$

\begin{equation}
\dot{|s_j|} \leq C\cdot |s_j|,
\end{equation}

which is not compatible with the fact that $s_j\rightarrow 0$.

\section{Uniform control for N solitons with separated poles}\label{sec:turbu}

\subsection{Separated-poles case}

In this subsection, we will show that in the case where the poles are separated, meaning that there exists $\eta>0$ such that for any $j \neq k$, $t>0$,

\begin{equation}
	|Re(x_k) - Re(x_j)| \geq \eta,
\end{equation}

infinite time blow up can not happen, even up to a subsequence.

\begin{theorem}
	Assume that the poles are separated.We define $J$ the set of indices $j$ such that there exists $t^j_n$,
	\begin{equation}
		\left\{
		\begin{aligned}
			&t_n^j \xrightarrow[n \rightarrow \infty]{} \infty, \\
			&Im(x_j)(t_n) \xrightarrow[n \rightarrow \infty]{} 0.
		\end{aligned}
		\right.
	\end{equation}

	We have already shown the absence of finite time blow up. In addition, we have the following estimate. For any $j \in J$, there exists $C_j>0$ and $D_j>0$ such that 
	
	\begin{equation}
		Im(x_j(t)) \geq C_j e^{-D_j t}.
	\end{equation}

\end{theorem}

We will first have to show the following lemma, that contains intermediate results and properties. We will establish the proofs in the same order. We consider a subsequence $t_n$ such that $Im(x_j)(t_n) \rightarrow 0$ for $j \in J$.

\begin{lemma}\label{unlemmepour}
	There exists $T>0$, $C>0$, such that for $t>T$, The following are satisfied.
	\begin{enumerate}
		\item For any $j$, we have that $|s_j| \leq \frac{C}{\min |Im(x_i)|}$.
		\item For any $j$, we have that $|s_j|(t_n) \xrightarrow[n \rightarrow \infty]{} 0$.
		\item For any $j$, we in fact have $|s_j| \leq C$.
	\end{enumerate}
\end{lemma}

\begin{proof}(Of lemma \ref{unlemmepour})

\begin{enumerate}
	\item We start with the fact that for any $x,t$, $|m(t,x)| = 1$ and write 
	
	\begin{equation}\label{uneequationencore}
		m(t,x) = m_0 + \sum_{k=1}^N \frac{s_k(t)}{ x - x_k(t)} + \sum_{k=1}^N \frac{\bar s_k(t)}{x-\bar x_k(t)}.
	\end{equation}

	We introduce $z_k$, $2N$ real numbers that we will choose later on. We consider the system consisting of $2N$ lines, that is obtained by plugging $z_k$ into the expression of $m(t,x)$.
	
	\begin{equation}
		\left\{
		\begin{aligned}
			&\sum_{k=1}^N \frac{s_k(t)}{z_1-x_k(t)} + \sum_{k=1}^N \frac{\bar s_k(t)}{z_1-\bar x_k(t)} = m(t,z_1) - m_0 = p_1(t), \\
			&\sum_{k=1}^N \frac{s_k(t)}{z_2-x_k(t)} + \sum_{k=1}^N \frac{\bar s_k(t)}{z_2-\bar x_k(t)} = m(t,z_2) - m_0 = p_2(t), \\
			&\dots\\
			&\sum_{k=1}^N \frac{s_k(t)}{z_{2n}-x_k(t)} + \sum_{k=1}^N \frac{\bar s_k(t)}{z_{2n}-\bar x_k(t)} = m(t,z_{2n}) - m_0 = p_{2n}(t),
		\end{aligned}
		\right.
	\end{equation}
	
	where $|p_{k}(t)| \leq 2$, for any $k \in \{1,\dots,2N\}$ and $t>0$. This system can be re-written as 
	
	\begin{equation}
		A \cdot S = P,
	\end{equation}
	
	where $A$ is a Cauchy matrix of the form 
	
	\begin{equation}
	(A)_{k,l} = \frac{1}{z_k - x_l(t)},
	\end{equation}

	for $k \in \{1, \dots , 2N\}$ and $l \in \{1,\dots,N\} $ and 
	
	\begin{equation}
		(A)_{k,l} = \frac{1}{z_k - \bar x_{l-N}(t)},
	\end{equation}

	for $k\in \{1 , \dots , 2N\}$ and $l\in \{N+1,\dots , 2n\}$, and $S$ is the vector
	
	\begin{equation}
		S = (s_1(t),s_2(t),\dots,s_N(t),\bar s_1(t),\bar s_2(t),\dots, \bar s_N(t)).
	\end{equation}

	Since $z_k$, $x_l$ and $\bar x_l$ are all different, the matrix $A$ is invertible and we have 
	
	\begin{equation}\label{expressioninverse}
		(A^{-1})_{i,j} = (B)_{i,j} = \frac{\prod_{k=1}^{2N}  (z_i-y_k)(z_k - y_j)}{(z_i - y_j) \left( \prod_{k\neq i}^{2N} (z_i-z_k) \right) \left( \prod_{k \neq j}^{2N} (y_j-y_k) \right)},
	\end{equation}

	where $y_k = x_k(t)$ for $k \leq N$, $y_k = \bar x_{k-N}(t)$ otherwise. Now, we look at the terms that could make this expression unbounded. We now choose $z_k$ as follows.
	
	\begin{equation}
		z_k = Re(x_k)(t),
	\end{equation}

	for $k \in \{1,\dots,N\}$ and 
	
	\begin{equation}
		z_k = Re(x_{k-N})(t)+ \frac{\eta}{5},
	\end{equation}
	
	for $k \in \{N+1.\dots,2N\}$. Now, by the separation assumption, we have that for $k \neq j$, and $j,k \in \{1,\dots,N\}$, 
	
	\begin{equation}
		|z_j-z_k| \geq \eta.
	\end{equation}
	
	For $j,k \in \{N+1,\dots,2N\}$, we have
	
	\begin{equation}
		|z_j - z_k| \geq \eta.
	\end{equation}
	
	Finally, for $j \in \{1,\dots,N\}$ and $k \in \{N+1,\dots,2N\}$,
	
	\begin{equation}
		|z_j-z_k| \geq \frac{\eta}{2}.
	\end{equation}
	
	We now look at (\ref{expressioninverse}). First, we fix $k \neq i$. We assume that $k,i \in \{1,\dots,N\}$. We have
	
	\begin{equation}
		\frac{|z_i-y_k|}{|z_i - z_k|} = \frac{|Re(x_i) - Re(x_k) - i Im(x_k)|}{ |z_i - z_k| }\leq \frac{|z_i - z_k| + |Im(x_k)|}{|z_i - z_k|} \leq 1 + \frac{2 |Im(x_k)|}{\eta}.
	\end{equation}
	
	Now, since all $|Im(x_i)|$ are all bounded by an absolute constant depending on the initial condition (By $I = \sum_{k} Im(x_i)(0))$, we conclude
	
	\begin{equation}
		\frac{|z_i-y_k|}{|z_i-z_k|} \leq 1 + \frac{C}{\eta}.
	\end{equation}
	
	We now consider  or $k,i \in \{N+1,\dots,2N\}$ and $k\neq i$. We have
	
	\begin{equation}
		\frac{|z_i - y_k|}{|z_i - z_k|} = \frac{|Re(x_i) + \frac{\eta}{5}  - Re(x_k) + i Im(x_k)| }{|Re(x_i) - Re(x_k)|} \leq 1+\frac{2 \eta}{5 \eta} + \frac{2|Im(x_k)|}{\eta}.
	\end{equation}
	
	We conclude that
	
	\begin{equation}
		\frac{|z_i - y_k|}{|z_i - z_k|} \leq 2 + \frac{C}{\eta}.
	\end{equation}

	Finally, we consider $i \in \{1,\dots,N\}$ and $k \in \{N+1, \dots,2N\}$ (the last case is identical). We have
	
	\begin{equation}
		\frac{|z_i - y_k|}{|z_i - z_k|} = \frac{| Re(x_i) - Re(x_k) + i Im(x_k) |}{|z_i - z_k|} \leq 2+\frac{C}{\eta}.
	\end{equation}
	
	The other term is identical. Hence, by regrouping the nominator and the denominator when $k \neq i,j$, we obtain
	
	\begin{equation}
		|(A^{-1})_{i,j}| = \frac{(z_i-y_i) (z_j-y_j)}{(z_i-y_j)} \cdot \left(2 + \frac{C}{\eta}\right)^{2N-2}.
	\end{equation}
	
	Now, we have by definition if $i,j \{1,\dots,N\}$ that
	
	\begin{equation}
		 \frac{(z_i-y_i) (z_j-y_j)}{(z_i-y_j)} = \frac{ Im (x_i) Im(x_j) }{|Re(x_i) - Re(x_j) - i Im(x_j)|},
	\end{equation}

	which is smaller than $C'/\varepsilon\leq C$. If $i,j \in \{N+1,\dots,2N\}$, we have
	
	\begin{equation}
		\left| \frac{(z_i-y_i) (z_j-y_j)}{(z_i-y_j)} \right| \leq C \frac{ (\eta + |Im (x_i)|) (\eta + |Im(x_j)|) }{|Re(x_i) - Re(x_j) - i Im(x_j)| + \eta} \leq C.
	\end{equation}
	
	If $i \in \{1,\dots, N\}$ and $j \in \{N+1,\dots,2N\}$, we have

	\begin{equation}
		\left| \frac{(z_i-y_i) (z_j-y_j)}{(z_i-y_j)} \right| \leq C \frac{ (|Im (x_i)|) (\eta + |Im(x_{j-N})|) }{|Re(x_i) - Re(x_{j-N}) - i Im(x_j)|} \leq C,
	\end{equation}

	except in the case where $j = i+N$, in this case, the real part of the denominator simplifies and one obtain
	
	\begin{equation}
	\left| \frac{(z_i-y_i) (z_j-y_j)}{(z_i-y_j)} \right| \leq C \frac{ (|Im (x_i)|) (\eta + |Im(x_{j-N})|) }{|Re(x_i) - Re(x_{j-N}) - i Im(x_j)|} \leq \frac{C}{|Im(x_{N-j})|}.
\end{equation}
	
	Now, the final case is $i \in \{N+1,\dots,2N\}$ and $j\in \{1,\dots,N\}$. We have in that case 
	
	\begin{equation}
	\left| \frac{(z_i-y_i) (z_j-y_j)}{(z_i-y_j)} \right| \leq C \frac{ (|Im (x_{i-N})| + \eta) (|Im(x_{j})|) }{|Re(x_i) - Re(x_{j-N}) - i Im(x_j)|+ \eta} \leq C.
\end{equation}
	
	Overall, all the coefficients of $B$ are bounded, except the ones of the form $(B)_{i,i+N}$ (for $i \in \{1,\dots,N\}$) in which case we only have 
	
	\begin{equation}
		|(B)_{i,N+i}| \leq \frac{C}{|Im(x_i)|}.
	\end{equation}
	
	Overall, since the operator norm of $B$ is (up to a constant depending only on the dimension) smaller than the sum of the coordinates of each row, we have
	
	\begin{equation}
		|||B||| \leq \frac{C}{|Im(x_{i_0})|},
	\end{equation}
	
	where $i_0$ is such that $|Im(x_{i_0}) = \min(|Im(x_i)|)$. Now, since 
	
	\begin{equation}
		S = B \cdot P,~|P| \leq C,
	\end{equation}

	we have that for any $j$,
	
	\begin{equation}
		|s_j| \leq \frac{C}{|Im(x_{i_0})|}.
	\end{equation}
	
	This concludes the proof of the first part of the lemma.
	
	\item We now want to show that for any $j$ such that there exists $t_n$, 
	
	\begin{equation}
		\left\{
		\begin{aligned}
			&t_n \rightarrow \infty,\\
			&Im(x_j)(t_n) \rightarrow 0,
		\end{aligned}
		\right.
	\end{equation}
	
	we have 
	
	\begin{equation}
		s_j(t_n) \rightarrow 0.
	\end{equation}
	
	We will admit this result, as the proof is identical to the one done in the previous section. The idea is the following. The energy $H^{1/2}$ is a preserved quantity, we express it either with the analytic or algebraic expression as

	\begin{equation}
		||m||_{1/2} = 2 \pi \sum_{j,k=1}^N \frac{s_j \cdot \bar s_k}{(x_j-x_k)^2},
	\end{equation}
	
	or
	
	\begin{multline}
		||m||_{1/2}^2 \\= C \int_{x=-\infty}^{+\infty} \left| \int_{y=-\infty}^{+\infty} \frac{\partial}{\partial x} \left( m_0 + \sum_{j=1}^N \left( i \frac{s_j}{x-x_j} - i \frac{\bar s_j}{x-x_j} \right) \frac{1}{\sqrt{|x-y|}} dy \right)  \right|^2 dx.
	\end{multline}
	
	Then, using the facts that $|x_j-x_k| \geq \eta$ for $j\neq k$ and $|s_j| \leq C |Im(x_{i_0})|^{-1}$, we show that the terms can not simplify each other, and hence each diagonal term must be bounded. We skip the details because we have already done both situations in details previously. 
	
	This leads to 
	
	\begin{equation}
		\frac{|s_j|^2}{|Im(x_j)|^2} \leq C.
	\end{equation}
	
	In particular, for each $j \in J$, we obtain $s_j(t_n) \rightarrow 0$. 
	
	\item We now want to obtain the following result. For any $t>0$, for any $j \in \{ 1,\dots,N \}$,
	
	\begin{equation}
		|s_j(t)| \leq C.
	\end{equation}
	
	In order to do so, we will apply a second time the Cauchy Matrix argument, but using now the intermediate results that we have obtained. In particular, every $s_j$ for $j\in J$ will be put on the "constant" side. Let us now write the precise argument.
	
	We define $z_k(t) = Re(x_k)(t)$ for $k \in \{1,\dots,N\}$, and $z_k(t) = Re(x_{k-N})(t)$ for $k \in \{N+1,\dots,2N\}$. Plugging $z_i \in \mathbb{R}$ into the expression of $m$, we obtain
	
	\begin{equation}
		m(z_i,t) - m_0 = \sum_{j=1}^N \frac{s_j}{z_i - x_j} + \sum_{j=1}^N \frac{\bar s_j}{z_i - \bar x_j}.
	\end{equation}
	
	These $2 N$ equations can be written as 
	
	\begin{equation}
		A \cdot S = P,
	\end{equation}
	
	where $S = (s_1,\dots,s_N,\bar s_1, \dots, \bar s_N)$, $P = (m(z_1,t),\dots,m(z_{2N},t))$ (and is bounded), and $A$ is the Cauchy matrix defined as 
	
	\begin{equation}
		(A)_{i,j} = \frac{1}{z_i - x_j},
	\end{equation}
	
	for $j \in \{1,\dots,N\}$ and 
	
	\begin{equation}
		(A)_{i,j} = \frac{1}{z_i - \bar x_j},
	\end{equation}
	
	for $j \in \{N+1,\dots,2N\}$.
	
	This is what we have done previously. Another way to write these equalities is to write
	
	\begin{equation}
	m(z_i,t) - m_0 - \sum_{j \in J}  \frac{s_j}{z_i - x_j} + \sum_{j\in J} \frac{\bar s_j}{z_i - \bar x_j} = \sum_{j\notin J} \frac{s_j}{z_i - x_j} + \sum_{j\notin J} \frac{\bar s_j}{z_i - \bar x_j}.
\end{equation}

	We will use the second result of this proof, which is that for $j \in J$,
	
	\begin{equation}
		\frac{|s_j|}{|Im(x_j)|} \leq C.
	\end{equation}
	
	Hence, the left-hand side is bounded, since for any $i,j$, $|z_i - x_j| \geq |Im(x_j)|$, since $z_i \in \mathbb{R}$. We now have a system of $2N$ equations, where we put $s_k$, $k \notin J$, on the "constant" side. We see it as a system of equations with $2N - 2 |J|$ variables, since we already know that $|s_j|$ is bounded for $j \in J$. We drop all the lines corresponding to $i \in J$, or $i-N \in J$, and obtain a system of the form
	
	\begin{equation}
		\tilde A \cdot \tilde S = \tilde P,
	\end{equation}
	
	where $\tilde S = (s_{j_1},\dots , s_{j_{|J|}},\bar s_{j_1},\dots, \bar s_{j_{|J|}})$ (with $J= \{j_1,\dots,j_{|J|}\}$), $\tilde P$ is bounded and $\tilde A$ is the Cauchy matrix defined as (for $j,j-N \notin J$, $i,i-N \notin J$),
	
	\begin{equation}
		(\tilde A)_{i,j} = \frac{1}{z_i - x_j},
	\end{equation}

	for $j \in \{1,\dots,N\}$ and $j \notin J$, and 
	
	\begin{equation}
		(\tilde A)_{i,j} = \frac{1}{z_i - \bar x_j},
	\end{equation}

	for $j \in \{N+1,\dots,2N\}$ and $j-N \notin J$. 

	Now, $|Im(x_j)| \geq \varepsilon >0$ for $j \notin J$. Hence, the coefficients of the inverse of $\tilde A$ satisfy 
	
	\begin{equation}
		\left| (\tilde A)^{-1}_{i,j} \right| \leq C.
	\end{equation}

	Since the operator norm of the inverse of $\tilde A$ is bounded by the sum of all the coefficients of each line, we conclude that $\tilde S$ is bounded since $\tilde P$ is bounded. Overall, we obtain that 
	
	\begin{equation}
		\left\{
		\begin{aligned}
			& |s_j| \leq C,~ j\notin J,\\
			&s_j \rightarrow 0,~ j \in J.
		\end{aligned}
		\right.
	\end{equation}
	
\end{enumerate}
	
\end{proof}

We now go on with the proof of the main result of this section. 

\begin{comment}
To do so, we will obtain a contradiction from the results we have established. The main argument, is that the conjunction of $|s_i|$ bounded and the time evolution equations of the spin is not compatible with the spins converging to $0$.
\end{comment}

\begin{proof}

We recall the time evolution equation of the spins, 

\begin{equation}\label{encoreunemdr}
	\dot s_j = -2 \sum_{k \neq j} \frac{s_j \times s_k}{(x_j-x_k)^2}. 
\end{equation}

	Since we have $|x_j - x_k| \geq \eta$, (\ref{encoreunemdr}) leads to
	
	\begin{equation}
		\left| \frac{\partial }{\partial t} \left| s_j(t) \right| \right| \leq C_j |s_j|,
	\end{equation} 

	and hence 
	
	\begin{equation}
		\ln \left( \frac{|s_j|(t)}{|s_j(t_0)|} \right) \geq -(t-t_0) C_j,
	\end{equation}

	which leads to
	
	\begin{equation}
		|s_j(t)| \geq A_j e^{- t D_j}.
	\end{equation}

	Since we have
	
	\begin{equation}
		\frac{|s_j|}{Im(x_j)} \leq C,
	\end{equation}

	we obtain the result.

\end{proof}

\subsection{Well posedness at $t=+\infty$, general case }

In this subsection, we prove the following theorem, which contains both the well-posedness for large times and for infinite times.

\begin{theorem}
    Assume that the speeds are different, i.e. $v_i\neq v_j$ for $i \neq j$, and $m$ be a solution of \eqref{HWM} of the form \eqref{expressionSol}. Then, there exists $B_1$ and $B_2$ such that for
    \begin{equation}
        \left\{
        \begin{aligned}
            &|s_j(t)| \leq B_1,~ j\in \{1,\dots,N\},\\
            &|Im(x_j(t))| \geq B_2,~ j \in \{1, \dots, N \}.
        \end{aligned}
        \right.
    \end{equation}
\end{theorem}

We will skip some details of the proof, as it is similar to what has been done for the case $|Re(x_j)-Re(x_k)| \geq \eta$, and focus on what is different. In particular, we will admit the following lemma. It is obtained using the conservation of the $H^{1/2}$ energy norm and of the norm of $m$. One can also apply the boundedness of the spin that we have shown when the poles are separated to this case since different propagation speeds will lead to separated poles for large enough times.

\begin{lemma}\label{lemme3}
    The poles are bounded, i.e there exist $B_1$ such that 
    \begin{equation}
        |s_j(t)| \leq B_1.
    \end{equation}
\end{lemma}

\begin{proof}
    The difference comes from the fact that the different speed gives a factor $(v_j - v_k) t$ in the dynamical system satisfied by the spins and poles. Indeed, \eqref{CM} now gives 

    \begin{equation}
        | \dot s_j(t) | \leq \frac{ 8 (N-1) B_1 |s_j(t)|}{t^2},
    \end{equation}

    which yields

    \begin{equation}
        \left| \frac{\ln(|s_j(t))|}{\ln(|s_j(t_0)|)} \right| \leq \frac{8 (N-1) B_1}{t_0},
    \end{equation}

    so 

    \begin{equation}
        |s_j(t)| \geq |s_j(t_0)| e^{- \frac{8 (N-1) B_1}{t_0}}.
    \end{equation}

    Using now the fact that, for the same reasons as previously,

    \begin{equation}
        \frac{|s_j(t)|}{|Im(x_j(t))|} \leq C,
    \end{equation}

    we obtain 

    \begin{equation}
        |Im(x_j(t))| \geq \frac{|s_j(t_0)|}{C}e^{- \frac{8 (N-1) B_1}{t_0}}.
    \end{equation}
    
\end{proof}

\printbibliography

\end{document}